\documentclass{amsart} 

\usepackage[active]{srcltx}

\newtheorem{definition}{Definition}[section]

\newtheorem{remark}[definition]{Remark}
\newtheorem{example}[definition]{Example}
\newenvironment{ack}{\noindent{\bf Acknowledgements}.}{}

\newtheorem{lemma}[definition]{Lemma}
\newtheorem{proposition}[definition]{Proposition}
\newtheorem{theorem}[definition]{Theorem}
\newtheorem{corollary}[definition]{Corollary}

\def\P{{\mathbb{P}}}
\def\A{{\mathbb{A}}}

\def\K{{\mathbb{K}}}
\def\kK{{\mathcal{K}}}
\def\N{{\mathbb{N}}}

\def\I{{\mathcal{I}}}

\def\T{{\underline{T}}}
\def\t{{\underline{t}}}

\def\X{{\underline{X}}}

\def\x{{\underline{x}}}
\def\u{{\underline{u}}}
\def\Z{{\mathbb{Z}}}

\def\C{{\mathrm{C}}}
\def\V{{\mathrm{V}}}

\def\F{{\underline{F}}}

\def\bdeg{{\mbox{bideg}}}
\def\aa{{\bf{a}}}

\def\cc{{\bf{c}}}
\def\dd{{\bf{d}}}
\def\ee{{\bf{e}}}

\begin{document}
\title{rational plane curves parameterizable by conics}

\author{Teresa Cortadellas Ben\'itez}
\address{Universitat de Barcelona, Departament d'{\`A}lgebra i Geometria.
Gran Via 585, E-08007 Barcelona, Spain}
\email{terecortadellas@ub.edu}

\author{Carlos D'Andrea}
\address{Universitat de Barcelona, Departament d'{\`A}lgebra i Geometria.
Gran Via 585, E-08007 Barcelona, Spain} \email{cdandrea@ub.edu}
\urladdr{http://atlas.mat.ub.es/personals/dandrea}
\thanks{Both authors are supported by the Research Project MTM2010--20279 from the
Ministerio de Ciencia e Innovaci\'on, Spain}


\begin{abstract}
We introduce the class of rational plane curves parameterizable by conics as an extension of the family of curves parameterizable by lines (also
known as monoid curves). 
We show that they are the image of monoid curves  via suitable quadratic transformations in projective plane. 
We also describe all the possible proper parameterizations of them, and a set of minimal generators of the Rees Algebra associated to these parameterizations, 
extending well-known results for curves parameterizable by lines.
\end{abstract}
\maketitle

\section{Curves parameterizable by forms of low degree}\label{uno}
This article deals with algebraic and geometric features of a special familiy of rational plane curves. 
Let $\K$ be an algebraically closed field. For a positive integer $k,$ we will denote with $\P^k$  the $k$-dimensional projective 
space over $\K$. Let $\C\subset \P^2$
be an algebraic plane curve of degree $d$, that is the the zero locus of
an irreducible homogeneous polynomial $E(X_1,X_2,X_3)\in\K[X_1,X_2,X_3]$ of
degree $d$.

A curve is \emph{rational} if it is birationally equivalent to $\P^1$, i.e.\
there exist dominant rational maps $\phi:\P^1 \to \C$ and $\psi:
\C\dasharrow \P^1$ such that $\psi \circ \phi =id_{\P^1}$ and $\phi
\circ \psi =id_{\C};$ equivalently there is an open subset of $\C$
 isomorphic to an open subset of $\P^1$
(or $\A^1$). If this is the case, the cardinality of the
general fiber of $\phi $ and $\psi $ is equal to one. So, $\phi$  
actually defines a \emph{proper} (i.e.\ generically injective)
parameterization of $\C$. Note that $\psi$'s \emph{domain} (the
largest set where the map is defined) coincides with the set of
nonsingular points of $\C$.

 A pair
$(F_1(X_1,X_2,X_3),\,F_2(X_1,X_2,X_3))$ of polynomials
in $\K[X_1,X_2,X_3],$ homogeneous of the same degree $d'$ without common factors
defines a birational map
\begin{equation}\label{inverse}
\begin{array}{cccc}
\psi:&\C&\dashrightarrow &\P^1\\
&(x_{1}:x_{2}:x_{3})&\mapsto&(F_1(x_{1},x_{2},x_{3}):F_2(x_{1},x_{2},x_{3}))
  \end{array}
 \end{equation}
 if  there exist a triple
$(u_1(T_1,T_2),\,u_2(T_1,T_2),\,u_3(T_1,T_2))$ of homogeneous
polynomials  without common factors in $\K[T_1,T_2]$
defining a parameterization of $C$ of the form
\begin{equation}\label{param}
\begin{array}{cccc}
\phi:&\P^1&\to&\C\\
&(t_{1}:t_{2})&\mapsto&(u_1(t_{1},t_{2}):u_2(t_{1},t_{2}):u_3(t_{1},t_{2})),
  \end{array} \end{equation}
with $\phi =\psi^{-1}$  as rational maps. Note that $\phi$ is globally defined but $\psi$ not necessarily. In fact, it is well-known
(see Lemma \ref{silverman}) that the set of singular points of $\C$ is contained in the algebraic variety defined by $F_1(\X)$ and $F_2(\X)$ in $\P^2$. 
Note that the inclusion may be strict, see for instance Example \ref{ejj}.   Set
$\t:=(t_1,t_2)$ and $\x:=(x_1,x_2,x_3)$.  The birationality of $\psi$ is equivalent to the following two claims:
\begin{equation}\label{u2}
(u_1(F_1(\x),F_2(\x)):u_2(F_1(\x),F_2(\x)):u_3(F_1(\x),F_2(\x))
)=(x_1:x_2:x_3)
\end{equation}
 for almost all $(x_1:x_2:x_3) \in \C $, and
\begin{equation}\label{u1}
 \big(F_1(u_1(\t),u_2(\t),u_3(\t)):F_2(u_1(\t),u_2(\t),u_3(\t))\big)=(t_1:t_2)\\
  \end{equation}
 for almost all $(t_1:t_2) \in \P^1.$ Note that the expressions in the left hand side of \eqref{u2} and \eqref{u1} are well-defined as the
families of polynomials are homogeneous.

Set $\T:=(T_1,T_2),\,\X:=(X_1,X_2,X_3),\,\u(\T)=(u_1(\T),\,u_2(\T),\,u_3(\T))$ and $\F(\X)=(F_1(\X),F_2(\X)).$ If \eqref{inverse} holds, we say
that $\C$ is \emph{parameterizable} by $\F(\X)$ (or by $\psi$)
and that $\u(\T)$ (or $\phi$) is the
proper parameterization induced by $\F(\X)$.
\medskip

Note that \eqref{u1} is equivalent to
\begin{equation}\label{u}
T_1F_2(\u(\T))-T_2F_1(\u(\T))=0,\end{equation}
and it turns out that \eqref{u} implies \eqref{u2}. This is clear 
 if the
characteristic of $\K$ is zero, and reasoning as in 
Proposition $2.1$ \cite{CD10} for $\psi$ birational, one gets the general case.

In order to find $\u(\T)$ starting from the data $\psi$ given in \eqref{inverse}, some geometry is needed. For a set of homogeneous elements 
$S\subset\K[\X]$, we denote with $\V(S)\subset\P^2$ the variety defined by it.
The fact that $\C=\V(E(\X))$ is parameterizable
by $\F(\X)$  means that the system
\begin{equation}\label{dprima}
\left\{\begin{array}{lcc}
E(\X)&=&0\\
T_1F_2(\X)-T_2F_1(\X)&=&0
         \end{array}
\right.
\end{equation}
has only one solution in $\P^2_{\overline{\K(\T)}}\setminus \P^2$ counted with multiplicities, or
equivalently 
has
 $dd'-1$ zeroes in $\P^2$ counted with
multiplicities. Here, $\P^2_{\overline{\K(\T)}}$ is the projective plane over $\overline{\K(\T)}$, the algebraic closure of $\K(\T)$. Note that our definition is not
the same as the one given in \cite[Definition 4.51]{SWP08} but a
more restrictive one as  shown in \cite[Theorem 4.54]{SWP08}.

\par 
From a computational point of view, a curve in the plane is typically given by either its implicit equation $E(\X)$ or ---if it is rational--- a parameterization
like \eqref{param}. Whether there exists a proper parameterization of $\C$ and, if this is the case, the computation of $\psi$ having as input $\phi$ or viceversa, 
are typical problems
of Computational Algebraic Geometry, see \cite{SWP08} and the references therein for more on the subject.

Let us consider the situation from a more algebraic perspective. Set $R:=\K[\T],$ and let $I$ be the ideal 
$\langle
u_1(\T),u_2(\T),u_3(\T) \rangle \subset R.$ The Rees
Algebra associated to $I$ is defined as $\mbox{Rees}(I):=\K[\T][I\,Z],$  where
$Z$ is a new variable. There is a graded epimorphism of $K[\T]$-algebras
defined by
\begin{equation}\label{ris}
\begin{array}{cccc}
   {\mathfrak h}:&\K[\T][\X]&\to&\mbox{Rees}({\mathcal I})\\
&X_i&\mapsto&u_i(\T)\,Z.
  \end{array}
\end{equation}
Set $\kK:=\ker ({\mathfrak h})$. Note that a description of $\kK$ allows also a full characterization of
$\mbox{Rees}({\mathcal I})$ via \eqref{ris}. This is why we call it \emph{the defining ideal of the Rees Algebra} associated to $\u(\T)$.
Condition \eqref{u}  is equivalent to the fact that
$T_1F_2(\X)-T_2F_1(\X)\in\kK.$ \par
Observe that $\kK$ is a bihomogeneous ideal, and that one has an identification of
$\kK_{\ast,1}$ with $\mbox{Syz}(I),$ the first module of syzygies of $I$.
It turns out that $\mbox{Syz}(I)$ is a free $R$-module of rank $2$ generated by two
elements, one of $\T$-degree $\mu$ for an integer $\mu$ such
that $0\leq\mu\leq\frac{d}{2}$, and the other of $\T$-degree $d-\mu$.
In the Computer Aided Geometric Design community, such a basis is called a $\mu$-basis of $I$ (see for instance \cite{CSC98,CGZ00,CCL05}). Indeed, by the
Hilbert-Burch Theorem, $I$ is generated by the maximal order
minors of a $3\times 2$ matrix $\varphi$ and  the
 homogeneous resolution of $I$ is
\begin{equation}\label{sec}
0\longrightarrow R(-d-\mu)\oplus R(-d-(d-\mu))
\stackrel{\varphi}{\longrightarrow} R(-d)^3
\stackrel{(u_1,u_2,u_3)}{\longrightarrow} I\longrightarrow  0.\end{equation}
This matrix is called the Hilbert-Burch matrix of $I$ and its
columns describe the $\mu$-basis.
\par Computationally, a $\mu$-basis provides simple (i.e.\  in both $(\T,\X)$-degrees) elements to describe the parameterization of $\C$
given in \eqref{param} than the data $\u(\T)$. The search for more simple elements to describe $\C$ leads to the study of 
the minimal generators of $\kK$. Indeed, the so-called method of implicitization by using moving curves of low degrees (described in \cite{SC95, SGD97, ZCG99}) 
is just a first step into a more complex picture which was described by Cox in \cite{cox08}, and subsequently worked out in 
\cite{CHW08,HSV08, KPU09, HSV09,bus09, HW10,CD10} among others. However, we are still far from being able to describe minimal generators of 
$\kK$ for a general ideal of a parametric plane curve $I$ as above. This paper is a contribution
in that direction. We will make a detailed study of rational curves parameterizable by forms of degree $2$, i.e.\  the situation $\deg(\F(\X))=2$
in \eqref{inverse}. 
The case of curves parameterizable by forms of degree $1$ has been completely described in \cite{cox08,bus09}.

\smallskip
Before starting, we present some results concerning existence and uniqueness of the polynomials $F_1(\X),\,F_2(\X)$ defining \eqref{inverse} for a fixed $\C$. 
Any rational plane curve $\C$ is parameterizable  by forms of degree
$d'$ for some $d'$. As a matter of fact, the method of adjoint curves proposed in \cite{walker} to parameterize any rational curve produces a map $\psi$ as 
in \eqref{inverse}, with $\F(\X)$ of degree less than or equal to $\deg(\C)-2$. The following result shows that if
$d'<\frac{\deg(\C)}{2}$, then not only $d'$ is unique but also the ideal
$\langle F_1(\X),\,F_2(\X)\rangle$.
\begin{proposition}\label{d'}
Let  $\C$ be a curve of degree $d$ parameterizable  by $(F_1(\X),\,F_2(\X))$, with $\deg(F_i(\X))=d'$.
Suppose that $\C$ is also parameterizable by $(F^0_1(\X),\,F^0_2(\X))$, the latter being forms of degree
$d_0'$ with both $d',\,d_0'<d$. Then, either $d'+d_0'\geq d$ or
$d'=d_0'$ and $\langle
F_1(\X),\,F_2(\X)\rangle=\langle F^0_1(\X),\,F^0_2(\X)\rangle$.
\end{proposition}

\begin{proof}
Let $\phi(\t):=\u(\t)$ and $\phi^
0(\t):=\u^0(\t),$ be the proper parameterizations
of $\C$ induced respectively by $(F_1(\X),\,F_2(\X))$ and $(F^0_1(\X),F^0_2(\X))$. Denote with $\psi$ the inverse of $\phi$.
Then, $\psi \circ \phi^0$ is an automorphism of $\P^1$, and hence there
exists a pair $(\rho_1(\T),\rho_2(\T))=:\rho(\T)$ of $\K$-linearly independent linear forms such that
$\u^0(\T)=\u(\rho(\T))$.

Let $\u^ 0(\T):=(u^0_1(\T), u^0_2(\T), u_3^0(\T))$ be the
parameterization induced by $(F^0_1,F^0_2)$. From \eqref{u} we have
$T_1F^0_2(\u(\rho(\T)))-T_2F^0_1(\u(\rho(\T)))=0.$ And
by writing  $T_1$ and $T_2$ as linear combinations of
$\rho_1(\T)$ and $\rho_2(\T)$, we get
\[
\rho_1(\T)F'_2(\u(\rho(\T)))-\rho_2(\T)F'_1(\u(\rho(\T)))=0
\]
with $\langle F^0_1(\X),\,F^0_2(\X)\rangle=\langle
F'_1(\X),\,F'_2(\X)\rangle$. As $\rho$ is an automorphism, we deduce
\[
T_1F'_2(\u(\T))-T_2F'_1(\u(\T))=0.\]

This equality, combined with $T_1F_2(\u(\T))-T_2F_1(\u(\T))=0$
implies that the polynomial $F'_1(\X)F_2(\X)-F'_2(\X)F_1(\X)$
vanishes on $\C$. As this is an element of degree $d'+d_0'$ and $\C$
has degree $d$, if $d'+d_0'<d$, then we have that
$$F'_1(\X)F_2(\X)-F'_2(\X)F_1(\X)=0.$$

Now, using the fact that $F_1(\X)$ and $F_2(\X)$ do not share any common factor,  we
deduce that $F_i(\X)$ divides $F'_i(\X)$ for $i=1,2$,  so $d'\leq
d_0'$ and
\[
\langle F_1^0(\X), F_2^0(\X) \rangle=\langle F_1'(\X), F_2'(\X)
\rangle\subset\langle F_1(\X),\,F_2(\X) \rangle .\] Applying 
the same argument symmetrically, we conclude that $d_0'\leq
d'$, and hence $$\langle F_1(\X),F_2(\X)\rangle\subset\langle
F^0_1(\X),\,F^0_2(\X)\rangle.$$
\end{proof}

\smallskip
If we restrict our attention to the set of curves parameterizable by forms of degree $d'$ for a fixed value of $d'$,
the following  natural questions arise:
\begin{itemize}
\item Can we describe geometrically all of them?
\item What does a proper parameterization of a curve
in this family look like?
\item Given $u_1(\T),u_2(\T),u_3(\T)\in K[\T]$ parametrizing  a plane curve parameterizable by forms of degree $d'$, can we describe  the
minimal homogeneous free resolution of 
 $\langle u_1(\T),\,u_2(\T),u_3(\T) \rangle $?
\item  Given $u_1(\T),u_2(\T),u_3(\T)\in K[\T]$ as above, can we describe a minimal set of generators
of $\kK$  ?
 \end{itemize}

An already interesting case is when $d'=1$. Such curves are called in \cite{SWP08}
\textit{parameterizable by lines}. Other authors call them \textit{monoid curves} (\cite{JLP08}). The answer to all these questions are well-known for them. 
We will review them along the text in order to compare
them with the main focus of this paper, which is $d'=2$. We will refer to them as \emph{curves parameterizable by conics}. In Section \ref{sconics}
we will describe all possible proper parameterizations of them, and also compute a non-trivial multiple of its implicit equation. Most of the time, this 
polynomial will actually be the one defining its implicit equation and, when it is not the case, the implicit equation will be given by its irreducible factor of
 largest degree (see Theorem \ref{complete}).
\par In Section  \ref{geometry}, we describe geometrically the space of all curves parameterizable by conics. In Theorem \ref{mtmt} we show that
they are the image of curves parameterizable by lines via a quadratic birational transformation of $\P^2$. Not surprisingly, 
the type of quadratic transformation depends on the geometry of the variety defined by $F_1(\X),\,F_2(\X)$ in $\P^2$.
\par Then we turn to study the last of the questions above. In Section \ref{s1} we present an extension of some of the tools used in \cite{CD10} for  
curves parameterizable by lines, to a
more general context. These extended tools will be used in Section \ref{Rees} to exhibit a complete set of generators of $\kK$ for proper parameterizations
of curves parameterizable by conics.  Curiously, the description of the generators depends on whether the degree of $\C$ is even or odd. In the first
case, a ``moving conic'' arising from the classical method of implicitization with the aid of moving curves comes into play (see Proposition \ref{mconic}).
\par It is worth mentioning here that the results in Sections \ref{s1} and \ref{Rees} are independent of the previous sections, so the reader interested  in
the questions related to the Rees Algebra can skip the first pages without harm. Of course it would be very interesting to get a further
understanding of the situation for $d'\geq3$, but our techniques only allow us to deal with curves parameterizable by conics. In Section \ref{conclu}, we conclude
with open questions and problems.

\medskip

\begin{ack}
We are very grateful to E.~Casas-Alvero for a careful reading of a preliminary draft of this paper, and also for patiently 
explaining us several features of the geometry of plane curves, in particular for helping us work out the 
canonical forms of Lemma \ref{reduccion}, and the quadratic
transformations  appearing in Section \ref{geometry}. We are also grateful to
J.~C.~Naranjo, J.~I.~Burgos and the anonymous referee for helpful comments and suggestions, and to the anonymous referee for several suggestions and corrections in the final version of this text. Our computations have been done with the aid of the softwares {\tt Macaulay 2}, {\tt Maple}, and 
{\tt Mathematica.}
\end{ack}

\bigskip
\section{Parameterizations and implicit equations of curves parameterizable by lines and conics}\label{sconics}
In this section we will explore algebraic aspects of curves
parameterizable by forms of degrees $1$ and $2$. They will be useful when studying geometric
properties of the singularities of these curves. The  case of curves parameterizable by lines is well-known in the literature. We review it here in
order to compare it with curves parameterizable by conics. Curves of degree $1$ (lines in $\P^2$) are easily to describe so we will assume from now on
that $d\geq2$.

\smallskip
\subsection{Curves parameterizable by lines}\label{monoid}
We start with the following result which characterizes curves parameterizable by lines having $(0:0:1)$ as a point of maximal multiplicity. Without loss of 
generality, we can assume that the inverse $\psi$ defined in \eqref{inverse} is given by $F_1(\X)=X_1,\,F_2(\X)=X_2.$

\begin{proposition}
Let $a(\T),\,b(\T)\in\K[\T]$ be homogeneous polynomials without common
factors, of degrees $d-1$ and $d>1$ respectively. Set
\begin{equation}\label{monn}
\left\{\begin{array}{ccl}
u_1(\T)&:=&T_1\,a(\T),\\
u_2(\T)&:=&T_2\,a(\T),\\
u_3(\T)&:=&b(\T).
         \end{array}
\right.
\end{equation}
Then, $\u(\T):=(u_1(\T),\,u_2(\T),\,u_3(\T))$ defines a proper
parameterization of curve $\C$ of degree $d$ parameterizable by lines having $(0:0:1)\in\C$ of multiplicity $d-1.$ Moreover, $b(X_1,X_2)-a(X_1,X_2)X_3$ is an
irreducible polynomial defining $\C.$  This curve is parameterizable by $(X_1,X_2)$. Reciprocally, any curve defined implicitly as
$b(X_1,X_2)-a(X_1,X_2)X_3=0$ in $\P^2$ with $a(\T),\,b(\T)$ as
above, is a curve parameterizable by lines with $(0:0:1)\in\C$ having multiplicity $d-1$.
\end{proposition}

\begin{proof}
Write $b(\T)=b_1(\T)T_1+b_2(\T)T_2.$ It is then easy to see that
the matrix
$$\varphi:=\left(\begin{array}{cc}
T_2& b_1(\T)\\
-T_1& b_2(\T)\\
0& -a(\T)
        \end{array}
\right)
$$
is the Hilbert-Burch matrix of the ideal $\langle
u_1(\T),\,u_2(\T),\,u_3(\T)\rangle\subset\K[\T]$, as in \eqref{sec}. By looking at
the $\T$-degree of the first column, we get that $\mu=1$, i.e.\  there is a
generator of the $\mbox{Syz}(I)$ of $\T$-degree one.
Proposition $2.1$ in \cite{CD10} tell us then that $\u(\T)$ defines a
birational map $\phi:\P^1\longrightarrow \C:=\phi (\P^1)$  whose inverse is given by $(X_1,X_2)$. In particular, $\phi$ is a
proper parameterization of a curve of degree $d$ having with $(0:0:1)\in\C$
having multiplicity $d-1.$ The fact that the implicit equation is
given by $b(X_1,X_2)-a(X_1,X_2)X_3$ was shown in \cite[Lemma
$2.5$]{CD10}.
\par
The rest of the proof follows straightforwardly: given $a(\T),\,b(\T)\in\K[\T]$ homogeneous without 
common factors and with respective degrees $d-1,\,d$. With this data we define
the parameterization \eqref{monn} and then we will find that the
implicit equation of $\C$ is given by the irreducible polynomial
$b(X_1,X_2)-a(X_1,X_2)X_3$.
\end{proof}

\smallskip
\subsection{Curves parameterizable by conics}\label{deg2}
In order to mimic the results obtained above, by making a linear
change of coordinates in $\P^2$ we start by assuming that
$(0:0:1)\in\V(\F(\X))$. Set ${\mathcal
F}(\T,\X):=T_1F_2(\X)-T_2F_1(\X)$, and write
\begin{equation}\label{conicg}
{\mathcal
F}(\T,\X)=l_1(\T)X_1X_2+l_2(\T)X_1X_3+l_3(\T)X_2X_3+l_4(\T)X_1^2+l_5(\T)X_2^2,
\end{equation}
with $l_i(\T)$ a homogeneous linear form in $\K[\T],\ i=1, 2,
3, 4, 5$.

\begin{proposition}\label{prop22} The conic defined by ${\mathcal F}(\T,\X)$ in
$\P^2_{\overline{\K(\T)}}$  is degenerate if and only if each
$F_i(\X)$ is the product of two linear forms in $\K[X_1,X_2]$. If this is 
the case, there is a curve $\C$ parameterizable by
$\F(\X)$ if and only if  $\C$ is either a line or parameterizable by lines.
\end{proposition}

\begin{proof}
If ${\mathcal F}(\T,\X)$ defines a degenerate conic then there exist ${\mathcal A}(\T,\X),\,{\mathcal B}(\T,\X)\in\overline{\K(\T)}[\X]$
homogeneous of $\X$-degree one such that
\begin{equation}\label{cua}
{\mathcal F}(\T,\X)={\mathcal A}(\T,\X)\,{\mathcal B}(\T,\X).
\end{equation}
As the left hand side has degree at most one in $X_3$, one of the factors on the right hand side do not depend on $X_3$. Suppose w.l.o.g.\ that 
$\deg_{X_3}({\mathcal A}(\T,\X))=0$, and write $$T_1F_2(\X)-T_2F_1(\X)={\mathcal F}(\T,\X)=Q(\T,X_1,X_2)+X_3\,L(\T,X_1,X_2),$$
with $Q(\T,X_1,X_2),\,L(\T,X_1,X_2)\in\K[\T,\X]$. If $L(\T,X_1,X_2)\neq0,$ then ---due to \eqref{cua}--- both 
polynomials $Q(\T,X_1,X_2)$ and $L(\T,X_1,X_2)$ will have
a non trivial common factor in $\overline{\K(\T)}[X_1, X_2]$. 
But this implies that they also share a common factor in $\K[\T,X_1, X_2]$, so a factorization as in  \eqref{cua} holds, with ${\mathcal A}(\T,\X)\,{\mathcal B}(\T,\X)\in \K[\T,X_1, X_2].$ Looking now at the degree in $\T$ in \eqref{cua}, we have that one of the two factors in the right hand side does not depend on $\T,$ which implies that $F_1(\X)$ and $F_2(\X)$ have a common factor of positive degree, a contradiction. Hence, $L(\T,X_1,X_2)=0$, which implies that $F_1(\X)$ and $F_2(\X)$ only depend on $X_1,\,X_2$, and 
they factorize as a product of linear forms, as $\K$ is algebraically closed.

The converse follows straightforwardly as $T_1F_2(X_1,X_2)-T_2F_1(X_1,X_2)$ factorizes as a product of two linear forms with coefficients in 
$\overline{\K(\T)},$ and hence they define a product of lines in $\P^2_{\overline{\K(\T)}}.$
\par
Now, suppose that
$F_1(\X),\,F_2(\X)\in\K[X_1,X_2]$. It is easy to see that here is a curve
parameterizable by these conics if and only if there is a solution
in $\P^1_{\K(\T)}$ of the equation $T_1F_2(X_1,X_2)-T_2F_1(X_1,X_2)=0$. By
dividing this equality by $X_2^2$, we get a quadratic equation in
$\frac{X_1}{X_2}$ whose coefficients are linear forms in $\T$. By
Gauss Lemma, any rational solution should have both numerator and
denominator being of $\T$-degree at most one. By looking at the shape of the first two coordinates of \eqref{monn}, we conclude that $\C$ is either a line or
parameterizable by lines.
\end{proof}

\begin{remark}\label{elunico}
If $T_1F_2(X_1,X_2)-T_2F_1(X_1,X_2)=0$ has no rational solutions in $\P^2_{\K(\T)},$
then there are no rational curves parameterizable by
$\F(\X)$. We will see below that this is actually the
only possible choice of a complete intersection of conics in $\P^2$ which does not
parameterize a curve $\C$.
\end{remark}

\medskip
Now we deal with nonsingular pencils of conics. We will
describe all the rational plane curves they produce by means of the
usual argument of cutting out the pencil with a moving line passing
through $(0:0:1)$. 

\begin{proposition}\label{paramparam}
Let $F_1(\X),\,F_2(\X)\in\K[\X]$ be homogeneous of degree $2$
without common factors such that  $(0:0:1)\in\V(\F(\X))$. If the
conic defined by ${\mathcal F}(\T,\X)$ in $\P^2_{\overline{\K(\T)}}$ is nondegenerate, then
for any pair $a(\T),\,b(\T)\in\K[\T]$ of homogeneous elements of
the same degree $d_0>1$ without common factors, the polynomials
\begin{equation}\label{paramm} \left\{
\begin{array}{ccl}
u_1(\T)&=&-a(\T)\big(a(\T)l_2(\T)+b(\T)l_3(\T)\big) \\
u_2(\T)&=&-b(\T)\big(a(\T)l_2(\T)+b(\T)l_3(\T)\big)\\
u_3(\T)&=& a(\T)b(\T)l_1(\T)+a(\T)^2l_4(\T)+b(\T)^2l_5(\T).
 \end{array}
\right.
\end{equation}
define a proper parameterization of a curve $\C$ parameterizable by
$\F(\X)$. Moreover, if $\gcd(X_1l_2(\F(\X))+X_2l_3(F(\X)),
a(\F(\X))X_2-b(\F(\X))X_1)=1,$ then $\gcd(\u(\T))=1$, and $\deg(\C)=2d_0+1$. Moreover, $a(\F(\X))X_2-b(\F(\X))X_1$ is an irreducible polynomial defining the
curve.
\end{proposition}

\begin{proof}
As $(0:0:1)$ is a rational point of the nondegenerate conic in $\P^2_{\overline{\K(\T)}},$ we
can describe all the other rational solutions by using a pencil of
lines passing through this point. In order to do that, given $a(\T),\,b(\T)\in\K[\T]$ homogeneous elements of
degree $d_0>1$ without common factors, consider
the system
$$\left\{\begin{array}{lcc}
{\mathcal F}(\T,\X)&=&0,\\
b(\T)X_1-a(\T)X_2&=&0.
        \end{array}
\right.
$$
It has two solutions in $\P^2_{\K(\T)}$, one of them being
$(0:0:1)$, so the other is also rational and by computing it
explicitly  we get that it is proportional to 
$\u(\T)$  in \eqref{paramm}.  As
$\gcd(a(\T),\,b(\T))=1$ and due to the fact that at least one
between $l_2(\T)$ and $l_3(\T)$ is not identically zero (this is because the conic
defined by ${\mathcal F}(\T,\X)$ in $\P^2_{\overline{\K(\T)}}$ is
nondegenerate), we then have that \eqref{paramm} defines the
parameterization of a rational plane curve $\C$, which turns out to be parameterizable by
$\F(\X)$.  Hence, the parameterization is proper.
\par\smallskip

Let $E(\X)\in\K[\X]$ be an irreducible polynomial defining $\C$. For $(x_1:x_2:x_3)\in \C$ we have
$b(\F(\x))x_1-a(\F(\x))x_2=0$, which implies that $b(\F(\X))X_1-a(\F(\X))X_2$ is a multiple of $E(\X).$
In order to show that they are equal, first we will prove that the
latter is not identically zero. Indeed, if this were the case,
then there would exist $C(\X)\in\K[\X],$ homogeneous of degree
$2d_0-1>0$ such that
$$\begin{array}{ccl}
   a(\F(\X))&=&C(\X)X_1,\\
   b(\F(\X))&=&C(\X)X_2.
  \end{array}
$$
As $C(\X)$ has positive degree, there are infinite points $(x_1:x_2:x_3)\in\P^2$
such that $C(\x)=0$. For those points we will have
$a(\F(\x))=b(\F(\x))=0,$ but as $a(\T)$ and $b(\T)$ do not have
common zeroes in $\P^1$, this then implies that the point
$(x_1:x_2:x_3)\in\V(\F(\X))$, which contradicts the fact that
$\V(\F(\X))$ is a complete
intersection (hence finite). This shows that $b(\F(\X))X_1-a(\F(\X))X_2\neq0$.
\par\smallskip

Suppose that $X_1l_2(\F(\X))+X_2l_3(F(\X)$ and
$a(\F(\X))X_1-b(\F(\X))X_2$ have no common factors. Choose
$(x_1:x_2:x_3)\in \P^2$ such that $b(\F(\x))x_1-a(\F(\x))x_2=0,$ with $(x_1:x_2:x_3)$
neither in $\V(\F(\X))$ nor in
$\V(X_1l_2(\F(\X))+X_2l_3(F(\X)))$. By hypothesis, we still have
an open set in $\V(a(\F(\X))X_2-b(\F(\X))X_1)$ to make such
choices. From the first condition, we get
$(x_1:x_2)=(a(\F(\x)):b(\F(\x)))$. From the second constraint we
deduce that $a(\F(\x))l_2(\F(\x))+b(\F(\x))l_3(\x)\neq0.$ So, by
using \eqref{paramm}, we have that
\[(x_1:x_2:x_3)=(u_1(\F(\x)):u_2(\F(\x)):u_3(\F(\x)))\]
and hence the point lies in the image of the
parameterization. This can be done in an open set of this curve,
and so it implies that $b(\F(\X))X_1-a(\F(\X))X_2$ defines
$\C=V(E(\X))$. Algebraically we have that  ---up to a nonzero
constant in $\K$--- there exists $\nu\in\Z_{>0}$ such that
\begin{equation}\label{dee}
b(\F(\X))X_1-a(\F(\X))X_2=E(\X)^\nu.
\end{equation}
The polynomial on the left hand side has degree $2d_0+1$. By
inspecting \eqref{paramm}, and using the fact that
$\gcd(a(\T),b(\T))=1$, we conclude that the degree of $\C$ (which is the degree of any proper parameterization of it) is
equal to $$2d_0+1-\deg(\gcd(\u(\T)))=d_0+i,$$ with $0\leq i\leq d+1$. Computing
degrees in \eqref{dee} we get
$$2d_0+1=\nu(d_0+i).$$
This diophantine equation in $(\nu,i)$ has only two solutions:
$\nu=1$ and $i=d_0+1$, i.e.\  there are no common factors, or
$\nu=3,\,i=0$, which can only be possible if $d_0=1$.

\end{proof}

\begin{remark}
A quick glance at (\ref{paramm}) may let the reader think that all
curves parameterizable by conics have odd degree, but this is not
always the case as $\deg(\gcd(\u(\T)))$ may be strictly positive. Also it is
not true that all the curves parameterized by (\ref{paramm}) pass
through the point $(0:0:1)$ as the following cautionary example
shows.
\end{remark}

\begin{example}\label{ejj}
Set $F_1(\X):=X_1X_2-X_1X_3,\,F_2(\X):=X_1X_2-X_2X_3.$ We then
have
$l_1(\T)=T_1-T_2,\,l_2(\T)=T_2,\,l_3(\T)=-T_1,\,l_4(\T)=l_5(\T)=0.$
Set also $a(\T):=T_1^2,\,b(\T):=T_2^2$. We get
$$\begin{array}{ccl}
X_1l_2(\F(\X))+X_2l_3(\F(\X))&=&X_1X_2(X_1-X_2),\\
b(\F(\X))X_1-a(\F(\X))X_2&=&X_1X_2(X_1-X_2)(X_3^2-X_1X_2),
  \end{array}
$$
and it is easy to see that the implicit equation of the curve
defined by this data is given by $X_3^2-X_1X_2,$ which is a smooth conic. Note that
$(0:0:1)$ is not a point of the curve.
\end{example}
\smallskip
Next we will show that the case presented in Example \ref{ejj} is
somehow unusual in the sense that if $d_0>2$, then any curve being
parameterized by \eqref{paramm} actually passes through the point
$(0:0:1)$ and moreover, if there is a common factor among the
three polynomials defining the parameterization, then it has
degree at most $2$. In order to show that,  we present first a
 ``canonical'' form of the sequence $\{F_(\X),\,F_2(\X)\}$
which will depend on the geometry of $\V(\F(\X))$.

\begin{lemma}\label{reduccion}
Let $F_1(\X),\,F_2(\X)$ be a sequence of homogeneous forms
of degree $2$ in $\K[\X]$ without common factors and such that the conic defined by
${\mathcal F}(\T,\X)$ is nondegenerate in
$\P^2_{\overline{\K(\T)}}$. Assume also that
$(0:0:1)\in\V(\F(\X))$. Then, after a linear change of
coordinates in $\P^2,$ we can assume:
\begin{equation}\label{four}
\F(\X)=(X_1X_2-X_2X_3,X_1X_3-X_2X_3) \quad \mbox{if}\ |\V(\F(\X))|=4,
\end{equation}
 \begin{equation}\label{three}
\F(\X)=(X_1X_2,X_1X_3-X_2X_3) \quad \mbox{if}\ |\V(\F(\X))|=3,
\end{equation}
\begin{equation}\label{twotwo}
 \F(\X)=(X_1^2,X_2X_3) \quad \mbox{if}\ |\V(\F(\X))|=2\end{equation}
 and each of the points in $\V(\F(\X))$ has multiplicity two,
\begin{equation}\label{twothree}
\F(\X)=(X_1^2-X_2X_3,X_1X_2) \quad \mbox{if}\ |\V(\F(\X))|=2\end{equation}
and one of the points in $\V(\F(\X))$ has multiplicity three,
\begin{equation}\label{one} \F(\X)=(X_1^2,\,X_2^2-X_1X_3) \quad \mbox{if}\ |\V(\F(\X))|=1.\end{equation} 
\end{lemma}

\begin{proof}
This classification is classic and well-known in Projective Geometry, see for instance \cite[Chapter VII]{SK52}.\footnote{Even though most of
the books in classic Projective Geometry deal with fields of characteristic zero, it is easy to see that the arguments leading to this classification
 are characteristic-free.} 
\end{proof}

\begin{proposition}\label{boundegree}
Assuming the same hypothesis and notations of Proposition
\ref{paramparam},  $\deg\big(\gcd(\u(\T))\big)\leq 3.$
\end{proposition}

\begin{proof}
Note that linear changes of coordinates in $\P^2$ amount to linear
combinations of the $u_i(\T)$'s with coefficients in $\K$ which
are invertible, i.e.\  one can use the canonical forms of the
polynomials $F_1(\X),\,F_2(\X)$ given by Lemma \ref{reduccion}
without changing $\gcd(\u(\T))$. Note also that, as $a(\T),b(\T)$
have no common factors, then
$$\gcd(\u(\T))=\gcd(l_2(\T)a(\T)+l_3(\T)b(\T),a(\T)b(\T)l_1(\T)+a(\T)^2l_4(\T)+b(\T)^2l_5(\T)).$$
In each of the cases described in Lemma \ref{reduccion} we explicit the
values of $l_i$ for $i=1,\dots ,5$ and bound the degree of the gcd.
\begin{itemize}
\item In  \eqref{four} we have $$l_4(\T)=l_5(\T)=0,\,
l_1(\T)=-T_2,\, l_2(\T)=T_1, \,l_3(\T)=T_2-T_1.$$ Hence,
$\gcd(\u(\T))=\gcd(a(\T)T_1+b(\T)(T_2-T_1),a(\T)b(\T)T_2)$, and from here we can conclude that $\gcd(\u(\T))$ divides
$T_1T_2(T_1-T_2)$.

\item In  \eqref{three} we have $$l_4(\T)=l_5(\T)=0,\,
l_1(\T)=-T_2,\,l_2(\T)=T_1,\,l_3(\T)=-T_1.$$ In this case,
$\gcd(\u(\T))=\gcd(a(\T)T_1-b(\T)T_1,a(\T)b(\T)T_2)$  divides
$T_1T_2$.

\item In \eqref{twotwo} we have $$l_1(\T)=l_2(\T)=l_5(\T)=0,
\,l_3(\T)=T_1,\,l_4(\T)=-T_2.$$ 
We get that $\gcd(\u(\T))=\gcd(b(\T)T_1,a(\T)^2T_2)$  divides $T_1T_2$.

\item In \eqref{twothree} we have $$l_2(\T)=l_5(\T)=0,\,
l_1(\T)=T_1,\, l_3(\T)=T_2,\, l_4(\T)=-T_2.$$ 
So, we deduce that
$\gcd(\u(\T))=\gcd(b(\T)T_2,a(\T)b(\T)T_1-a(\T)^2T_2)$  divides
$T_2$.

\item In  \eqref{one} we have $$l_1(\T)=l_3(\T)=0,\,
l_2(\T)=-T_1,\,l_4(\T)=-T_2,\,l_5(\T)=T_1,$$ and we get that
$\gcd(\u(\T))=\gcd(a(\T)T_1,b(\T)^2T_1-a(\T)^2T_2)$ divides $T_1$.
\end{itemize}
In all of the cases, we get $\deg(\gcd(\u(\T)))\leq 3,$ which proves the claim.
\end{proof}

Now we can prove a complete version of Proposition
\ref{paramparam}.
\begin{theorem}\label{complete}
Let $F_1(\X),\,F_2(\X)$ be a sequence of quadratic
 forms in $\K[\X]$ without common factors such that
$(0:0:1)\in\V(\F(\X))$ and ${\mathcal F}(\T,\X)$ defines a nondegenerate conic in $\P^2_{\overline{\K(\T)}}.$ For any $a(\T),\,b(\T)\in\K[\T]$
homogeneous of degree $d_0>2$ without common factors, either
$a(\F(\X))X_2-b(\F(\X))X_1$ is an irreducible polynomial or it has
a unique irreducible factor of degree larger than $1$. In both
cases, this irreducible factor defines a rational curve
$\C\subset\P^2$ parameterizable by $\F(\X)$ and passing through
$(0:0:1)$. All the linear extraneous factors define equations of lines passing through the points of $\V(\F)$, and the degree of this factor is less
than or equal to three.
\end{theorem}
\begin{proof}
As shown in Proposition \ref{paramparam}, the pair $(a(\T),b(\T))$
defines the parameterization \eqref{paramm} of a curve $\C$ parameterizable
by $\F(\X)$. As $d_0>2,$ we then have $d_0+1>3$ and on
the other hand if there is a nontrivial $\gcd(\u(\T))$ in
\eqref{paramm}, its degree -thanks to Proposition
\ref{boundegree}- cannot be larger than three. This shows that
the factor $a(\T)l_2(\T)+b(\T)l_3(\T)$ cannot be completely cancelled when removing the $\gcd$ in \eqref{paramm}, and hence
$(0:0:1)$ is in the image of the parameterization. So, $\C$ passes through this point.

If  $\gcd(\u(\T))=1$, as the parameterization is proper,
$a(\F(\X))X_2-b(\F(\X))X_1$ has the same degree as the curve $\C$. Hence, it is the irreducible polynomial defining
it. Suppose then that this is not the case. Then there exist
$H(\X)\in\K[X]$ homogeneous and coprime with $E(\X)$ such that
\begin{equation}\label{deee}
a(\F(\X))X_2-b(\F(\X))X_1=E(\X)^\mu\,H(\X),
\end{equation}
with $\mu\in\N,\,E(\X)$ being the irreducible polynomial defining
$\C$. Let us say that
$\deg(E(\X))=\varepsilon,\,\deg(H(\X))=\rho>0$. By computing
degrees in \eqref{deee}, we get
$$2d_0+1=\mu\,\varepsilon+\rho
$$
Thanks to Proposition \ref{boundegree}, we know that $2d_0-2\leq
\varepsilon\leq 2 d_0+1,$ so we have $\mu(2d_0-2)+\rho\leq
2d_0+1$. As $d_0>2,$ we can conclude from here that $\mu=1$.
Moreover, we get that $\rho\leq 3,$ i.e.\ the degree of the
extraneous factor $H(\X)$ is bounded. It remains to show that $H(\X)$ decomposes as a product of linear factors. The proof of Proposition
\ref{paramparam} actually shows that 
$$\V\big(a(\F(\X))X_2-b(\F(\X))X_1\big)\subset\V(E(\X))\cup\V(X_1l_2(\F(\X))+X_2l_3(\F(\X))),$$
and hence the factors of $H(\X)$ must be among the factors of $X_1l_2(\F(\X))+X_2l_3(\F(\X))$. One can show that in all the possible cases listed in
Lemma \ref{reduccion}, the polynomial $X_1l_2(\F(\X))+X_2l_3(\F(\X))$ factorizes as a product of linear forms. Moreover, these linear forms can always 
be chosen in the set $\{X_1,\,X_2,\,X_1-X_2\},$ which are always lines passing through the points of $\V(\F).$
\end{proof}

\bigskip
\subsubsection{Examples}\label{example}
Let $d_0\in\N$ and set $a(\T)=T_1^{d_0},\,b(\T):=T_2^{d_0}$. We will consider all the possible scenarios given by Lemma \ref{reduccion}.
\begin{itemize}
\item For $\F(\X)=(X_1X_2-X_2X_3,X_1X_3-X_2X_3)$,  \eqref{paramm} becomes
$$ \left\{\begin{array}{ccl}
u_1(\T)&=&-T_1^{d_0} (T_1^{1 + d_0} - T_1 T_2^{d_0} + T_2^{1 + d_0})\\
u_2(\T)&=&-T_2^{d_0} (T_1^{1 + d_0} - T_1 T_2^{d_0} + T_2^{1 + d_0})\\
u_3(\T)&=&-T_1^{d_0} T_2^{1 + d_0}.
\end{array}\right.
$$ 
Note that $\gcd(\u(\T))=1$, hence $\C$ has degree $2d_0+1$. Computing explicitly the implicit equation we get
$$E(\X)=X_2^{d_0+1} (X_1 - X_3)^{d_0} - X_1X_3^{d_0} (X_1 - X_2)^{d_0}.$$

\item Set now $\F(\X)=(X_1X_2,X_1X_3-X_2X_3)$. The family $\u(\T)$ of \eqref{paramm} is now
$$ \left\{\begin{array}{ccl}
u_1(\T)&=&-T_1^{1 + d_0} (T_1^{d_0} - T_2^{d_0})\\
u_2(\T)&=&-T_1 T_2^{d_0} (T_1^{d_0} - T_2^{d_0})\\
u_3(\T)&=&-T_1^{d_0} T_2^{1 + d_0}.
\end{array}\right.
$$ 
Note that $\gcd(\u(\T))=T_1$ in this case, and hence $\deg(\C)=2d_0.$ 
Indeed, an explicit computation shows that
$$a(\F(\X))X_2-b(\F(\X))X_1=X_1\big(X_1^{d_0-1}X_2^{d_0+1}-X_3^{d_0}(X_1-X_2)^{d_0}\big),$$
hence the implicit equation is defined by $X_1^{d_0-1}X_2^{d_0+1}-X_3^{d_0}(X_1-X_2)^{d_0}.$
Note that in this case $$\gcd(X_1l_2(\F(\X))+X_2l_3(F(\X)),
a(\F(\X))X_2-b(\F(\X))X_1)=X_1,$$
(cf.\ Proposition \ref{paramparam}).

\item Set now $\F(\X)=(X_1^2,X_2X_3)$.Then,
$$ \left\{\begin{array}{ccl}
u_1(\T)&=&-T_1^{1 + d_0} T_2^{d_0}\\
u_2(\T)&=&-T_1 T_2^{2 d_0}\\
u_3(\T)&=&-T_1^{2 d_0} T_2,
\end{array}\right.
$$ 
with $\gcd(\u(\T))=T_1T_2$. Hence, $\deg(\C)=2d_0-1$ and computing explicitly $a(\F(\X))X_2-b(\F(\X))X_1$ we get
that it is equal to $X_1X_2\,E(\X)$, with 
$$E(\X)=X_1^{2d_0-1} -  X_2^{d_0-1} X_3^{d_0}.$$

\item For $\F(\X)=(X_1^2-X_2X_3,X_1X_2)$, we have  
$$\left\{\begin{array}{ccl}
u_1(\T)&=&-T_1^{d_0} T_2^{1+d_0}\\
u_2(\T)&=&-T_2^{1+2 d_0}\\
u_3(\T)&=&-T_1^{d_0+1} T_2(T_1^{d_0-1}  - T_2^{d_0-1}),
\end{array}\right.
$$ 
with $\gcd(\u(\T))=T_2$. So, $\deg(\C)=2d_0$ and $a(\F(\X))X_2-b(\F(\X))X_1$ is equal to $X_2\,E(\X)$ with
$$E(\X)=  (X_1^2 - X_2 X_3)^{d_0} -X_1^{1+d_0} X_2^{d_0-1}.$$

\item Finally, consider $\F(\X)=(X_1^2,\,X_2^2-X_1X_3).$ By computing explicitly, we get 
$$\left\{\begin{array}{ccl}
u_1(\T)&=&T_1^{1+2d_0} \\
u_2(\T)&=&T_1^{1+ d_0}T_2^{d_0}\\
u_3(\T)&=&T_1T_2(T_2^{2d_0-1}-T_1^{2d_0-1}).
\end{array}\right.
$$ 
Here, we have $\gcd(\u(\T))=T_1$. Again we get $\deg(\C)=2d_0$ and
$$a(\F(\X))X_2-b(\F(\X))X_1=X_1\,E(\X)$$ with
$E(\X)=  X_1^{2d_0-1}X_2 - (X_2^2 - X_1 X_3)^{d_0}.$
\end{itemize}

\bigskip

\section{The geometry of curves parameterizable by conics}\label{geometry}
In this section, we will study geometric properties of plane curves parameterizable by conics. We will show that essentially they are the image of  
a curve parameterizable by lines via a quadratic transformation of the plane.

\subsection{Quadratic transformations in the plane}
\begin{definition}\label{qt}
A rational map $\Lambda:\P^2\dasharrow\P^2$ is called a \emph{quadratic transformation} if $\Lambda$ is birational and there exist 
$Q_1(\X),\,Q_2(\X),\,Q_3(\X)\in\K[\X]$ homogeneous of degree $2$ without common factors such that 
\begin{equation}\label{Lambda}\Lambda(x_1:x_2:x_3)=\big(Q_1(\x):Q_2(\x):Q_3(\x)\big).
\end{equation}
\end{definition}
One of the most well-known of these quadratic transformations is the following 
\begin{equation}\label{cremona}
\Lambda_{\bf0}(x_1:x_2:x_3)=(x_2x_3:x_1x_3:x_1x_2),
\end{equation} which is used for
desingularization of curves, see for instance \cite{walker}. Even though there are birational automorphisms of $\P^2$ defined by homogeneous forms of arbitrary degree, we will focus here
in those of degree $2,$ as they will be crucial when studying curves parameterizable by conics.

\begin{proposition}\label{quadratic}
Let $F_1(\X),\,F_2(\X)\in\K[\X]$ be a sequence of homogeneous forms of degree $2$ without common factors. If the conic defined by ${\mathcal F}(\T,\X)$ 
in $\P^2_{\overline{\K(\T)}}$ is nondegenerate, then there
exists $F_3(\X)\in\K[\x]$ homogeneous of degree $2$ such that 
\begin{equation}\label{Lambdaf}
\begin{array}{cccc}
\Lambda_F:&\P^2&\dasharrow&\P^2\\
&(x_1:x_2:x_3)&\mapsto&\big(F_1(\x):F_2(\x):F_3(\x)\big)    
  \end{array}
\end{equation}
is a quadratic transformation. Moreover, $\Lambda_F^{-1}$ is also a quadratic transformation.
\end{proposition}

\begin{remark}
In characteristic zero, it is well-known that a birational transformation given by polynomials of degree $n$ has an inverse also given by forms of the same degree,
see for instance \cite{alberich}. 
\end{remark}

\begin{proof}
 We will use the canonical forms given in Lemma \ref{reduccion} in order to make explicit the polynomial $F_3(\X)$ in each of the possible cases. 
\begin{enumerate}
 \item If $|\V(\F(\X)|\geq3,$ we can suppose w.l.o.g.\ that $$\{(1:0:0),\,(0:1:0),\,(0:0:1)\}\subset\V(\F(\X)),$$ and hence by using \eqref{four} or \eqref{three},
it is easy to see that if we set $F_3(\X):=X_2X_3$, $\Lambda_F$ is actually the classical
transformation $\Lambda_{\bf0}$ composed with an automorphism of $\P^2$. As $\Lambda_{\bf0}^{-1}=\Lambda_{\bf0},$ it is easy to see that $\Lambda_F^{-1}$ 
can be defined with  linear combinations of
$X_1X_2,\,X_1X_3,\,X_2X_3$, hence it is a quadratic transformation.

\item If $|\V(\F(\X))|=2$, each point with multiplicity two, then by using \eqref{twotwo} we can assume w.l.o.g.\ that 
$$\F(\X)= (X_1^2,\,X_2X_3).$$
We set $F_3(\X):=X_1X_2$ and get  
$$\Lambda_F^{-1}(y_1:y_2:y_3)=\big(y_1y_3: y_3^2: y_1y_2\big),
$$
hence $\Lambda_F$ is birational with quadratic inverse.

\item If $|\V(\F(\X))|=2$ and one of the points in this set has multiplicity three, then by \eqref{twothree} we can assume after a linear change of coordinates that
$$\F(\X)=
( X_1^2-\,X_2X_3, X_1X_2).$$
Setting $F_3(\X):=X_2 X_3$ we get that
$$\Lambda_F^{-1}(y_1:y_2:y_3)=\big(y_2(y_1+ y_3):y_2^2:y_3(y_1+ y_3)\big).$$
Hence, $\Lambda_F$ is birational and the inverse is quadratic, as claimed.

\item If $\{(0:0:1)\}=\V(\F(\X))$. We then use \eqref{one} and suppose w.l.o.g.\ that
$$\F(\X)=( X_1^2,\,X_2^2-X_1X_3).$$
Once more, by setting $F_3(\X):=X_1X_2$, we get
$$\Lambda_F^{-1}(y_1:y_2:y_3)=\big(y_1^2: y_1y_3:y_3^2-y_1y_2\big),
$$ so we conclude that $\Lambda_F$ is birational with quadratic inverse. This completes the proof.
\end{enumerate}
\end{proof}

\begin{lemma}\label{zulema}
For any curve $\C_0$ of degree $d^0>1$ parameterizable by lines, having $(0:0:1)\in\C_0$ with multiplicity $d^0-1$, and any quadratic transformation $\Lambda:\P^2\dasharrow\P^2$
 whose inverse is defined by a sequence of quadratic forms $(F_1(\X),\,F_2(\X),\,F_3(\X)),\ \overline{\Lambda(\C_0)}$ is a curve 
parameterizable by  $(F_1(\X),F_2(\X))$.
\end{lemma}

\begin{proof}
Set $\C=\overline{\Lambda(\C_0)}.$  The fact that $\C_0$ is not a line implies that $\dim(\C)=1$. As $\C_0$ is parameterizable by $(X_1,X_2)$, then it is easy to verify then that $\C$ is parameterizable by $(F_1,F_2)$. 
\end{proof}

\begin{remark}
We are not claiming in Lemma \ref{zulema} that the first two coordinates of a quadratic transformation have a non trivial common factor. Also, it is not
necessarily true that  the image of a curve parameterizable by lines via a quadratic transformation cannot be a parameterizable by lines anymore. For instance, 
$\Lambda_{\bf0}$ has 
$\F(\X)=(X_2X_3,\,X_1X_3)$ which has $X_3$ as a common factor. Also, if $\C_0$ is any curve parameterizable by lines having its singularity at $(0:0:1)$, then it is easy to check
that $\overline{\Lambda_{\bf0}(\C_0)}$ is again a curve parameterizable by lines having its singularity at the same point.
\par Moreover, not necesarily the first two coordinates of a quadratic transformation define a polynomial ${\mathcal F}(\T,\X)$ whose set of zeroes in
$\P^2_{\overline{\K(\T)}}$ is a nondegenerate conic, for instance $\Lambda(x_1:x_2:x_3):=(x_1^2:x_1x_2:(x_1+x_2)x_3)$ is a quadratic transformation 
with inverse $\Lambda^{-1}(x_1:x_2:x_3)=(x_1(x_1+x_2):x_2(x_1+x_2):x_1x_3)),$ but the conic defined by $T_2X_1^2-T_1X_1X_2$ is degenerate  according to
Proposition \ref{prop22}. 
\end{remark}
\smallskip
Now we proceed to compare the degrees of $\C_0$ and its transform $C=\overline{\Lambda(C_0)}$.
We start with the following result, which will be of use in the sequel.
\begin{lemma}\label{auxxi}
Let $Q_1(\X),\,Q_2(\X),\,Q_3(\X)\in\K[\X]$ be a sequence of homogeneous quadratic forms such that $\Lambda$ defined 
in \eqref{Lambda} is a quadratic transformation,
and $\C\subset\P^2$ any curve of degree $d$. Let $\C_Q$ be a generic conic in the linear system defined by $Q_1(\X),\,Q_2(\X),\,Q_3(\X).$
Then, for  any point $p\in\C\cap\C_Q,$ we have $$m_p(\C\cap\C_Q)\leq d.$$
Moreover, the inequality is strict if $\C_Q$ and $\C$ do not have a common tangent at $p$.
\end{lemma}
\begin{proof}
Suppose w.l.o.g.\ that $p=(0:0:1).$ A generic linear combination of the $Q_i(\X)$'s must have a non-zero linear term with respect to $X_3$ otherwise those three
polynomials would depend only on $X_1$ and $X_2$ contradicting the fact that $\Lambda$ is a birational. This implies that $m_p(\C_Q)=1$. On the other hand,
we always have $m_p(\C)<d$. If $\C$ and $\C_Q$ intersect transversally at $p,$ then we have (cf.\ \cite[Proposition $3.6$]{HKT08})
$$m_p(\C\cap\C_Q)=m_p(\C)< d.
$$ 
In case they do not intersect transversally, as $\C_Q$ has a tangent line $L_Q$ having multiplicity one at $p$, then we have
$$m_p(\C\cap\C_Q)=m_p(\C\cap L_Q)\leq d,
$$
the last inequality is due to B\'ezout's Theorem applied to $\C$ and $L_Q$.
\end{proof}

\begin{proposition}\label{num}
With notations and assumptions as in Lemma \ref{zulema}, denoting with  $D^0$  the degree of $\overline{\Lambda(\C_0)}$, then we have
$d^0-1\leq D^0\leq 2d^0,
$ 
and the inequalities are sharp.
\end{proposition}
\begin{proof}
As before, set $\C=\overline{\Lambda(\C_0)}.$ Its degree can be computed as
the cardinality of $\C\cap L$, with $L$ a generic line in $\P^2,$ which we will choose as intersecting $\C$ in the (dense) open set of $\P^2$ where
$\Lambda$ is bijective. As $\Lambda$ is birational, then we can compute this intersection number via $\Lambda^{-1}$.
Then, $\C$ gets converted into $\C_0$ and $L$ in a generic linear combination of the quadratic polynomials $F_1(\X),\,F_2(\X),\,F_3(\X)$. 
We use then B\'ezout's Theorem 
in order to count the number of intersections between $\C_0$ and the conic $\Lambda^{-1}(L)$ to get
\begin{equation}\label{numeromania}
2d^0=D^0+\sum_{p\in\V(F_1,F_2,F_3)}m_p(\C_0\cap\Lambda^{-1}(L)).
\end{equation}
As the data $(F_1(\X),\,F_2(\X),\,F_3(\X))$ defines a birational transformation, it is easy to see that $|\V(F_1,F_2,F_3)|\leq 3.$ Moreover,
the scheme of points defined by $\F(\X)$ in $\P^2$ must have degree less than or equal to three, otherwise one of these polynomials
would be a linear combination of the others contradicting the fact that $\Lambda$ is a quadratic transformation.
\par We have in addition that 
$\C_0$ is parameterizable by lines. This implies that there is one point of multiplicity $d^0-1$ and the remaining have multiplicity one. Hence, thanks
to Lemma \ref{auxxi}, we have
$$0\leq \sum_{p\in\V(F_1,F_2,F_3)}m_p(\C_0\cap\Lambda^{-1}(L))\leq \left\{\begin{array}{l}
1+1+(d^0-1)\\
1+d^0,\end{array}
\right.=d^0+1.
$$
The first case is when $\V(\F(\X))$ has three points, hence there cannot be fixed tangential conditions in the linear system and this implies that we
can choose the generic line in such a way that $\Lambda^{-1}(L)$ cuts transversally $C_0$; the second case is when the linear system defined by
$\F(\X)$ has a fixed tangential condition. But then, we have that $\V(\F(\X))$ cannot have more than two points, and by using Lemma \ref{auxxi} we are done.
From here plus \eqref{numeromania}, we get the bounds of the claim.

Now we will show that the bounds are sharp. For a generic quadratic transformation $\Lambda,$
we will have $\sum_{p\in\V(F_1,F_2,F_3)}m_p(\C_0\cap\Lambda^{-1}(L))=0$. Indeed, one only has to pick $(F_1,F_2,F_3)$ 
in such a way that $\V(\F(\X))\cap\C_0=\emptyset.$  So, the
inequality at the left is generically  an equality. In order to show that the other inequality can also become an equality, let $d_0>1$ and consider the following parameterization
$$\left\{\begin{array}{ccl}
u_1(\t)&=&t_1\,\alpha(\t),\,\\
u_2(\t)&=&t_2\,\alpha(\t),\,\\
u_3(\t)&=&t_1t_2\beta(\t)   
  \end{array}\right.
$$
with $\alpha(\T),\,\beta(\T)$ homogeneous of degrees $d_0-1$ and $d_0-2$ without common factors and also without common factors with neither $T_1$ nor $T_2$.
Then, the curve $\C_0$ defined as the image of this parameterization is parameterizable by lines of degree $d_0$ with $p=(0:0:1)$ having multiplicity
$d_0-1.$
Consider $\Lambda_{\bf0}$ defined in \eqref{cremona}. Then, an straightforward computation shows that a proper parameterization of $\Lambda_0(\C_0)$ is given by
$$\left\{\begin{array}{ccl}
v_1(\t)&=&t_2\,\beta(\t),\,\\
v_2(\t)&=&t_1\,\beta(\t),\,\\
v_3(\t)&=&\alpha(\t);   
  \end{array}\right.
$$
i.e.\  $\Lambda_0(\C_0)$ is a curve of degree $d_0-1$. Note that this curve is either a line or again parameterizable by lines.
\end{proof}

\medskip
We can now describe geometrically the plane curves parameterizable by conics via quadratic transformations of curves parameterizable by lines. Recall that
thanks to Proposition \ref{prop22}, if $T_1F_2(\X)-T_2F_1(\X)$ defines a degenerate conic in $\P^2_{\overline{\K(\T)}},$ then any curve parameterizable by
$\F(\X)$ is either a line or parameterizable by lines. Also, curves of degree  $2$ are parameterizable by lines.
\begin{theorem}\label{mtmt}
Let $F_1(\X),\,F_2(\X)$ be sequence of homogeneous forms of degree $2$ without common factors such that $(0:0:1)\in\V(\F(\X))$ and 
$T_1F_2(\X)-T_2F_1(\X)$ does not define
a degenerate conic in $\P^2_{\overline{\K(\T)}}$.  
Consider any quadratic transformation of the form $\Lambda_F$ defined in \eqref{Lambdaf}. A curve $\C$ such that $\deg(\C)\geq3$ is parameterizable by $\F(\X)$ 
if and only if there exist $\C_0$ parameterizable by lines having $(0:0:1)$ as its only singular point and $\overline{\Lambda_F(\C)}=\C_0$.
\end{theorem}

\begin{proof}
As $T_1F_2(\X)-T_2F_1(\X)$ defines a nondegenerate conic in $\P^2_{\overline{\K(\T)}}$, we can find a quadratic transformation $\Lambda_F$ as in Proposition \ref{quadratic}. 
Set $\C_0$ to be the Zariski closure of $\Lambda_F(\C)$ in $\P^2$. By Proposition \ref{num}, we have that $\deg(\C_0)\geq \frac{\deg(\C)}{2}>1$, hence $\C_0$ is
not a line. We can also verify easily that $\C_0$ is parameterizable by $(X_1,X_2)$, hence it 
is parameterizable by lines and having $(0:0:1)\in\C_0$ with maximal multiplicity. 
\par In order to prove the converse, if we start with $\C_0$ as in the hypothesis and 
define $\C$ to be the Zariski closure of $\Lambda_F(\C_0),$ we can easily verify that $\C$ is parameterizable by $\F(\X)$. 
\end{proof}
\smallskip

\bigskip

\subsection{On the singularities of curves parameterizable by conics}
There is an increasing interest in the analysis of singularities of rational curves by means of elements of small degree in the Rees Algebra of
the parameterization, see for instance \cite{CKPU11}. Theorem \ref{mtmt} above shows that curves parameterizable by conics are only ``one quadratic 
transformation away'' from curves parameterizable by lines, and in principle it may seem that the study of their singularities can be done straightforwardly, as for instance the
transformation $\Lambda_{\bf0}$ defined in \eqref{cremona} is the one used in the process of desingularization of curves. The main drawback here is that ---as Theorem \ref{mtmt} 
claims--- a curve parameterizable by conics is the image of a curve parameterizable by lines with singular point in $(0:0:1)$ via \emph{any} quadratic transformation, and 
$\Lambda_{\bf0}$ is known to ``behave
properly'' if the curve is in a general position with respect to the coordinate axes (cf.\ the notions of ``good'' and ``excellent'' 
positions in \cite{ful69}). So, even if we use $\Lambda_{\bf0}$ to transform $\C$ into a curve parameterizable by lines, we cannot expect to get a straightforward dictionary between the only singularity of the 
curve parameterizable by lines and those of $\C$.
The analysis of the singularities of these curves require a further study of properties of general quadratic transformations, which goes beyond the scope
of this article.  

One case which is easy to tackle is when $|\V(\F(\X))|=4$, We will show that  in this situation, all of the four points are multiple points of  
$\C$ and moreover, there are no infinitely near multiple points.
We start by analyzing the only singularity of a curve parameterizable by lines.

\begin{proposition}\label{geomonoid}
Let $\C$ be a curve parameterizable by lines having $(0:0:1)\in\C$ with multiplicity $\deg(C)-1$, and implicit equation given by the polynomial
$b(X_1,X_2)+X_3\,a(X_1,X_2)\in\K[\X]$, with $a(\T),\,b(\T)$ homogeneous elements of degrees $d-1$ and $d$ respectively. Write
$$a(\T)=\cc_0\prod_{j=1}^{\tau}(\dd_jT_2-\ee_jT_1)^{\nu_j},$$
with $\cc_0\in\K\setminus\{0\},\,(\dd_j:\ee_j)\neq(\dd_k:\ee_k)$ if $j\neq k$, and $\nu_j\in\N$ for $j=1,\ldots, s$. 
Then,
\begin{enumerate}
\item there are $\tau$ different branches of $\C$ passing through $(0:0:1)$; 
\item denote with $\gamma_j$  the  branch of $\C$ at $\phi((\dd_j:\ee_j))$, here $\phi(t_1:t_2)$  is the parameterization of $\C$ given by \eqref{monn}. 
The tangent
to $\gamma_j$ at $(t_1:t_2)=(\dd_j:\ee_j)$ is the line $\dd_jX_2-\ee_jX_1=0$. In particular, different branches have different tangents (i.e.\  there are no tacnodes);
\item the order of contact of $\C$ with the tangent line $\dd_jX_2-\ee_jX_1=0$ at $(0:0:1)$ is equal to $\nu_j+1$;
\item the multiplicity of $(0:0:1)$ in $\C$ is $d-1$, and there are no infinitely near multiple points of $\C$. 
\end{enumerate}
\end{proposition}

\begin{proof}
The first three items follow straightforwardly from working out the parameterization \eqref{monn} in a neighborhood of the zeroes of $a(\T),$ plus the fact that for 
this proper parameterization we have $T_2u_1(\T)-T_1u_2(\T)=0$. 
\par In order to conclude, recall that $(0:0:1)$ is a point of multiplicity $d-1$. The genus formula shows that there cannot be no more singular 
points in $\C$.
\end{proof}

The following result  about curves and rational maps is well-known. We record it here for the convenience of the reader. Denote with $\mbox{Sing}(C)$ the set of singular points of $\C$ in $\P^2$.
\begin{lemma}\label{silverman}
If $\C$ is parameterizable by $(F_1(\X),F_2(\X)),$ then
 $\mbox{Sing}(\C)\subset \V(\F(\X)).$
\end{lemma}
\begin{proof}
Let $\phi$ be as in \eqref{param} a proper parameterization of $\C$ having as its inverse $\psi=(F_1:F_2)$ whenever it is defined, as in \eqref{inverse}.
As $\phi$ is defined on the whole $\P^1$, from $\psi\circ\phi=id_{\P^1},$ we have
\begin{equation}\label{identidad}
\big(F_1(\underline{\phi}(t_1:t_2)):F_2(\underline{\phi}(t_1:t_2))\big)=(t_1:t_2)\ \ \mbox{for } \underline{\phi}(t_1:t_2)\notin\V(\F).
\end{equation}
If $p=\phi(t_{01}:t_{02})\in\C$ is a singular point, and suppose that $(F_1(p):F_2(p))=(t_{01}:t_{02}),$ We then have two possible scenarios:
\begin{itemize}
 \item If the gradient of $\phi$ at $(t_{01}:t_{02})$ is equal to zero, by differentiating both sides of \eqref{identidad} and specializing 
$(t_1:t_2)\mapsto(t_{01}:t_{02})$ we would get a contradiction.
\item If the gradient of $\phi$ is not zero at $(t_{01}:t_{02})$, then there must be another point $(t_{11}:t_{12})\in\P^1$ such that $\phi(t_{11}:t_{12})=p$
(i.e.\ the curve ``passes'' at least twice over $p$). But then, we will have
$$\big(F_1(\underline{\phi}(t_{11}:t_{12})):F_2(\underline{\phi}(t_{11}:t_{12}))\big)=(t_{01}:t_{02})\neq(t_{11}:t_{12}),
$$
a contradiction with \eqref{identidad}.
\end{itemize}
This shows that for such a singular point $p\in\C, \,\F(p)=(0,0)$ which proves the claim.
\end{proof}

\begin{theorem}\label{4points}
Let $F_1(\X),\,F_2(\X)$ be a sequence of forms of degree $2$ in $\K[\X]$ without common factors, such that the conic defined by $T_1F_2(\X)-T_2F_1(\X)$ is
nondegenerate in $\P^2_{\overline{\K(\T)}}$. If $\C$ is parameterizable by $\F(\X),$ not parameterizable by lines, and $\V(\F(\X))$ has four points, 
then $\V(\F(\X))=\mbox{Sing}(\C)$ and at each $p\in \V(\F(\X))$, $p\in C$ is locally isomorphic to the singular  point of a curve 
parameterizable by lines. Thus $p$ is not a tacnode and has no infinitely near singular points. Hence 

\begin{equation}\label{aguap}
\sum_{p\in \V(\F(\X))}m_p(\C)(m_p(\C)-1)=(d-1)(d-2).
\end{equation}
\end{theorem}

\begin{proof}
After a linear change of coordinates, we may assume that we are in the conditions of \eqref{four} and hence 
$$\{(1:0:0),\,(0:1:0),\,(0:0:1),\,(1:1:1)\}=\V(\F(\X)).$$ 
As in the proof of Proposition \ref{quadratic}, by setting $F_3(\X):=X_2X_3,$ we get a quadratic transformation $\Lambda_F$ which is actually the 
composition of $\Lambda_{\bf0}$  with an automorphim of $\P^2$. It is easy to check that $p_0=(1:1:1)$ is not in the union of lines where $\Lambda_F$
is not invertible. Hence, in  neighborhood of this point, $\Lambda_F$ is actually an algebraic isomorphism. As $\overline{\Lambda_F(\C)}$ 
is not a line (due to the fact that $\deg(\C)>2,$ otherwise it would be parameterizable by lines), and is  parameterizable by 
$(X_1,X_2)$ with only singularity in $(0:0:1)=\Lambda_F(p_0)$, then properties (1) to (3) in Proposition \ref{geomonoid} apply to $p_0$ with respect to 
$\C,$ due to the fact that $\Lambda_F$ is a local isomorphism around $p_0$ and its image. For the same reason, the fact that there are no infinitely near 
multiple points of $\overline{\Lambda_F(\C))}$ above $\Lambda_F(p_0)$ (this is property (4) in Proposition
\ref{geomonoid}) implies that there cannot be infinitely near multiple points of $\C$ above $p_0$. 
\par
Making a linear change of coordinates, the role of $(1:1:1)$ can be played by the other three points of $\V(\F(\X))$, and this implies the claim for the other 
three points, i.e.\  we have shown $\V(\F(\X))\subset\mbox{Sing}(\C)$). The other inclusion follows by Lemma \ref{silverman}, hence we
have the equality. As there cannot be more singular points, and none of the elements in $\V(\F(\X))$ has infinitely near multiple points of $\C$ above it,
\eqref{aguap} follows due to the genus formula.
\end{proof}

\begin{remark}
 Note that the Theorem does not claim that the four singular points have the same multiplicity and character. Just that they ``look like'' (locally)
like a multiple point in a curve parameterizable by lines. This curve is not necessarily the same for all the points, as the following example shows.
\end{remark}

\begin{example}
Let $\C$ be the rational curve of degree $5$ 
defined by the polynomial $E(\X)=X_2^3(X_1-X_3)^2-X_1X_3^2(X_1-X_2)^2$ (this is the first bullet of Example \ref{example} with $d_0=2$). Its four singular points
are $(1:0:0),\,(0:1:0),\,(0:0:1)$ and $(1:1:1)$. By analyzing them explicitly, we get that
\begin{itemize}
 \item $(0:0:1)$ is an ordinary triple point;
\item $(0:1:0)$ and $(1:0:0)$ are cusps;
\item $(1:1:1)$ is an ordinary triple point.
\end{itemize}
We can straightforwardly verify equality \eqref{aguap} in this case:
$$12=(5-1)(5-2)=3\times2+2\times1+2\times1+2\times1.$$
\end{example}

We have thus completed our study of the singularities of  $\C$ in the case $|\V(\F(\X))|=4,$ which is somehow the generic case among curves parameterizable by conics.  Now we turn into the question of how the singularities look like in the remaining cases. 
We will see in Section \ref{s1} (Corollary \ref{maxmu}) that the value of $\mu$ in \eqref{sec} is always equal to $\lfloor \frac{d}2\rfloor$. This information is enough 
to show that if $|\V(\F(\X))|\leq 3,$ then $\C$ has always infinitely near singular points if $\deg(\C)>6$.
\begin{proposition}
If $\C$ is parameterizable by a sequence of forms $(F_1(\X),\,F_2(\X))$ of degree $2$ without common factors, with $d=\deg(\C)>6$, and $|\V(\F(\X))|\leq3,$ 
then $\C$ has infinitely near singular points.
\end{proposition}

\begin{proof}
If there are no infinitely near multiple points, due to the genus formula, we will have
$$(d-1)(d-2)=\sum_{p\in\C}m_p(\C)\big(m_p(\C)-1\big).
$$
In \cite[Theorem 1]{CWL08}, it is shown that there can only be one multiple point of multiplicity larger than $\mu$. Moreover, if this is the case, then the
multiplicity of this point is actually $d-\mu$. Suppose then that $|\V(\F(\X))|\leq3.$ As $\mbox{Sing}(\C)\subset\V(\F(\X))$ (cf.\ Lemma \ref{silverman}), we then
conclude that there are at most $3$ singular points. One of them has its multiplicity bounded by $d-\mu\leq \frac{d+1}{2}$ and the other two have both
multiplicities bounded by $\frac{d}{2}$. Hence, we get
$$\begin{array}{ccl}
0&=&\sum_{p\in\C}m_p(\C)\big(m_p(\C)-1\big)-(d-1)(d-2)\\ \\
&\leq &(\frac{d+1}{2}\frac{d-1}{2}+2\ \frac{d}{2}\frac{d-2}{2})-(d-1)(d-2)\\ \\ 
&=&-\frac{d^2-8d+9}{4}.
\end{array}
$$
For $d\geq7$, the last expression is negative. This concludes the proof.
\end{proof}

\bigskip
\section{The Rees Algebra of a rational parameterization}\label{s1}
Now we turn to the problem of computing a set of minimal generators for the presentation of the Rees Algebra
associated to the ideal of a rational parameterization of a curve parameterizable by conics. This section may be considered  an extension
of the results given in \cite{CD10} (see also \cite{bus09,CHW08})
for curves parameterizable by lines.

Let $I$ be the ideal of $\K[T_1,T_2]$ generated by three
homogeneous polynomials
$u_1(T_1,T_2),\,u_2(T_1,T_2),\,u_3(T_1,T_2)$ of degree $d$ without
common  factors. Recall that
$\mbox{Rees}(I)=\K[\T][I\,Z]$ is the Rees Algebra associated to
$I.$ Let $\kK\subset R[\X]$ be the kernel of the graded morphism of
$\K[\T]$-algebras ${\mathfrak h}$ defined in \eqref{ris}.
It is
a bigraded ideal (with grading given by total degrees in $\T$ and
$\X$) characterized by
$$P(\T,\X)\in \kK_{i,j}\iff \bdeg(P)=(i,j)\ \mbox{and} \ P(\T,\u(\T))=0.$$
Let $\phi:\P^1\to\P^2$ be the map given by (\ref{param}), and  set as before
$\C:=\phi(\P^1)$. 
As we observed in Section \ref{uno}, $\phi$ admits a rational
inverse $\psi:\C\dasharrow\P^1$ if and only if there exists an irreducible nonzero element
in $\kK_{1,*}:=\oplus_{j=0}^\infty \kK_{1,j}.$ Moreover, if $F_1(\X),\,F_2(\X)$ are coprime elements in $\K[\X]$ and
$T_1F_2(\X)-T_2F_1(\X)\in\kK_{1,\nu},$
then
$\F(\X)$ defines the  inverse of $\phi$.

In the terminology of \cite{cox08,bus09}, $\kK$ is the
\emph{moving curve ideal} of the parameterization $\phi$. An
element in $\kK_{\ast,j}$ is called a \emph{moving curve} of degree $j$
that follows the parameterization. In this sense, moving lines that follow the parameterization are
the elements of $\kK_{\ast,1}$ and there is an obvious isomorphism
of $\K[\T]$-modules
\begin{equation}\label{corresp}
 \begin{array}{ccc}
\kK_{*,1}&\to&\mbox{Syz}(\I)\\
a(\T)X_1+b(\T)X_2+c(\T)X_3&\mapsto&\big(a(\T),b(\T),c(\T)\big).
  \end{array}
\end{equation}

Recall from Section \ref{uno} that the first module of syzygies of $I$ is a free
$\K[\T]$-module generated by two elements, one in degree $\mu$ for
a positive integer $\mu$ such that $0<\mu\leq\frac{d}{2}$, and the
other of degree $d-\mu$. Such a basis is called a $\mu$-basis. In the sequel, we
will denote with $p_{\mu ,1}(\T,\X),\,q_{d-\mu ,1}(\T,\X)\in
\kK_{*,1}$ a (chosen) set of two elements in $\mbox{Syz}(\I)$ which are a basis of this module.

Note that with this language, we can say that there exists an
irreducible element in $\kK_{1,1}$ if and only if $\C$ is
parameterizable by lines, and this is equivalent also to $\mu =1$. We will see
(for $d>3$) that if there exists an irreducible element in
$\kK_{1,2}$ (that is  $\C$ is parameterizable by conics and not by
lines, cf.\ Proposition \ref{d'}) then  $\mu =\lfloor\frac{d}{2}\rfloor$. Before that, we
present two results that will be useful in the sequel.

The first of them is the analogue of Proposition $2.6$ in
\cite{CD10}.

\begin{proposition}\label{multiple}
Suppose $T_1F_2(\X)-T_2F_1(\X)\in\kK$ is an irreducible
polynomial. Then, $P_{i,j}(\T,\X)\in\kK_{i,j}$ if and only if
$P_{i,j}(F_1(\X),F_2(\X),\X)$ is a multiple of $E(\X)$.
\end{proposition}
\begin{proof}
We only have to show that $P_{i,j}(F_1(\X),F_2(\X),\X)$ vanishes
on $\C$ if and only if $P_{i,j}(\T,\X)\in\kK_{i,j}$. Taking into
account that $\C=\{\u(\t) \mid (t_1:t_2)\in
\P^1\},$ and that $(t_1:t_2)=(F_1(\u(\t)):F_2(\u(\t))$  for almost all
$(t_1:t_2)\in \P^1$, then
\[
\begin{array}{l}
\\
P_{i,j}(F_1(\x),F_2(\x),\x)=0 \textrm{ for all } (x_1:x_2:x_3)\in
\C \Leftrightarrow \\ \\
\Leftrightarrow P_{i,j}(F_1(\u(\t)),F_2(\u(\t)),\u(\t))=0
\textrm{ for all } (t_1:t_2)\in \P^1 \Leftrightarrow \\ \\
\Leftrightarrow P_{i,j}(F_1(\u(\t)),F_2(\u(\t)),\u(\t))=0
\textrm{ for almost all } (t_1:t_2)\in \P^1 \Leftrightarrow \\ \\
\Leftrightarrow P_{i,j}(t_1,t_2,\u(\t))=0
\textrm{ for almost all } (t_1:t_2)\in \P^1\Leftrightarrow\\ \\
\Leftrightarrow P_{i,j}(t_1,t_2,\u(\t))=0
\textrm{ for all } (t_1:t_2)\in \P^1\Leftrightarrow
P_{i,j}(\T,\X)\in\kK_{i,j}.
\end{array}\]
\end{proof}

The following proposition is the analogue of Lemma $2.7$ in
\cite{CD10}.

 \begin{lemma}\label{lem1}
Suppose $F_1(\X),\,F_2(\X)$ are homogeneous of degree $j_0$. Let
$P(\T,\X)$ be a bihomogeneous polynomial of bidegree
$(i,j)\in\N^2$, with $i>0,\,j\geq j_0$. Then there exists
$Q(\T,\X)$ bihomogeneous of bidegree $(i-1,(i-1)j_0+j)$ such that
\begin{equation}\label{connect}
F_2(\X)^iP(\T,\X)-T_2^iP(F_1(\X),F_2(\X),\X)=(T_1F_2(\X)-T_2F_1(\X))Q(\T,\X).
\end{equation}
\end{lemma}
\begin{proof}
$$
\begin{array}{l}
F_2(\X)^iP(\T,\X)-T_2^iP(F_1(\X),F_2(\X),\X)=\\
=P(T_1F_2(\X),T_2F_2(\X),\X)-P(T_2F_1(\X),T_2F_2(\X),\X).
\end{array}
$$
By applying  on the polynomial
$p(\theta):=P(\theta,T_2F_2(\X),\X)$  the first order Taylor
formula, the claim follows straightforwardly.
\end{proof}

\begin{proposition}
Assume that $\phi$ defined as in \eqref{param}, is a proper parameterization of a curve of degree $d$. Let $\nu$ be the
degree of a homogenous pair of polynomials in $\K[\X]$ defining the inverse of $\phi$. If $\mu>1$ then
$$\mu \, \nu +1\geq d.$$
\end{proposition}

\begin{proof}
Let $p_{\mu,1}(\T,\X)$ be a nonzero element in $\kK_{\mu,1}$. Due
to Proposition \ref{multiple}, we have that
$p_{\mu,1}(F_1(\X),F_2(\X),\X)$ is a multiple of $E(\X)$, which has degree $d$. As $\deg(p_{\mu,1}(F_1(\X),F_2(\X),\X))=\mu \, \nu +1$, it turns out that if
$\mu \, \nu +1<d,$ then  $$p_{\mu,1}(F_1(\X),F_2(\X),\X)=0.$$ By
(\ref{connect}), we then have $F_2(\X)p_{\mu,1}(\T,\X)\in\langle
T_1F_2(\X)-T_2F_1(\X)\rangle,$ which is a prime ideal and clearly
$F_2(\X)$ does not belong to it. So we conclude that
$p_{\mu,1}(\T,\X)$ is a multiple of $T_1F_2(\X)-T_2F_1(\X)$
which is impossible unless $\deg(F_1(\X))=\deg(F_2(\X))=1$ which
is equivalent to $\mu=1$.
\end{proof}

\begin{corollary}\label{maxmu}
If $\nu=2$ (i.e.\  $\phi$ parameterizable by conics) and there are
no linear syzygies, then $\mu=\lfloor\frac{d}{2}\rfloor$, the
maximum possible value.
\end{corollary}

It was shown already in \cite{bus09} that for $\mu\geq2$ the
description of generators of $\kK$ is much more complicated than
in the case of curves parameterizable by lines, so there is little  hope that the elementary
methods applied in \cite{CD10} can be used in these cases. Next we will show that instead of looking at low degrees for the syzygies of $\phi$, if
we try low degrees for the inverse of $\phi$, that the approach of \cite{CD10} can be adapted, and indeed produces a minimal set of 
generators of rational plane
curves parameterizable by conics (i.e.\ , the degree of the inverse is equal to $2$). We start by recalling the following:

\begin{proposition}[Proposition 3.6 in \cite{BJ03}]\label{1.1} 
The sequence $p_{\mu,1}(\T,\X),\,q_{d-\mu, 1}(\T,\X)$ is 
regular in $\K[\T,\X]$ and
\[
 \kK= \bigcup_{n\geq 0} \langle p_{\mu,1}(\T,\X),\,q_{d-\mu,
1}(\T,\X)\rangle:\langle T_1,T_2 \rangle^n.\]
\end{proposition}

As explained in \cite[Section $2$]{bus09}, in order to search for a set of generators of $\kK$, it is enough to consider forms of $\T$-degree lower than
$d$. Our next result is a refinement of this bound, which
essentially states that  we can replace $d-1$ by $d-\mu$.

\begin{theorem}\label{mu}
Let $u_1(\T),\,u_2(\T),\,u_3(\T)\in\K[\T]$ be homogeneous polynomials of
degree $d$ having no common factors. A minimal set of
generators of $\kK$ can be found with all its elements having
$\T$-degree  strictly less than $d-\mu$ except for the
generators of $\kK_{\ast,1}$ with $\T$-degree $d-\mu$.
\end{theorem}

\begin{proof}
Let $P(\T,\X)\in\kK_{i,j}$ with $i\geq d-\mu$, and
$\{p_{\mu,1}(\T,\X),\,q_{d-\mu,1}(\T,\X)\}$ as above,
a $\K[\T]$-basis of $\kK_{*,1}$. Let $L_\mu(\X)$ (resp. $M_{d-\mu}(\X)$) be
the coefficient of $T_2^\mu$ (resp. $T_2^{d-\mu}$) in
$p_{\mu,1}(\T,\X)$ (resp. $q_{d-\mu,1}(\T,\X)$). Also, let $W(\X)$
be the coefficient of $T_2^i$ in $P(\T,\X)$. As
$P(\T,\X)\in\kK_{i,j},$  due to Proposition \ref{1.1} we have
that there exists $a\in\N,\, \alpha(\T,\X),\,\beta(\T,\X)\in\K[\T,\X]$ such that
\begin{equation}\label{t2}
T_2^aP(\T,\X)=\alpha(\T,\X)p_{\mu,1}(\T,\X)+\beta(\T,\X)q_{d-\mu,1}(\T,\X).
\end{equation}
We set $T_1=0$ in
\eqref{t2}, and get an expression of the form
$$W(\X)=A(\X)L_{\mu}(\X)+B(\X)M_{d-\mu}(\X),
$$
with $A(\X),\,B(\X)\in\K[\X]$. Set then
\begin{equation}\label{recur}
Q(\T,\X):=P(\T,\X)-T_2^{i-\mu}A(\X)p_{\mu,1}(\T,\X)-T_2^{i-d+\mu}B(\X)q_{d-\mu,1}(\T,\X)
\end{equation}
By setting $T_1=0$ in \eqref{recur}, it is easy to see that
$Q(\T,\X)$ vanishes, so we have that
$$Q(\T,\X)=T_1\,\tilde{Q}(\T,\X)$$ with $\tilde{Q}(\T,\X)\in\kK_{i-1,j}$. If $i-1\geq d-\mu$, we have then that 
$$P(\T,\X)\in\langle \tilde{Q}(\T,\X),\,p_{\mu,1}(\T,\X),\,q_{d-\mu,1}(\T,\X)\rangle\subset\langle \cup_{\ell\leq i-1}\kK_{\ell,j}\rangle;$$
and by iterating this argument with $\tilde{Q}(\T,\X)$ instead of $P(\T,\X)$, we
conclude that $P(\T,\X)\in\langle\cup_{\ell\leq d-\mu}\kK_{\ell,j}\rangle$.
\par
If $i=d-\mu$, reasoning as above we arrive to
$$P(\T,\X)\in\langle\cup_{\ell\leq d-\mu-1}\kK_{\ell,j}\rangle+\langle p_{\mu,1}(\T,\X),q_{d-\mu,1}(\T,\X)\rangle,$$
and hence the claim follows.
\end{proof}

\section{The Rees Algebra of curves parameterizable by conics}\label{Rees}
All along this section we will assume that $\phi$ is
parameterizable by conics and not by lines, (i.e.\ $d>3$, see Proposition \ref{d'} ). Let
$(F_1(\X),F_2(\X))$ be the pair of forms of degree $2$ without
common factors defining the inverse of $\phi$. Then, due to
Corollary \ref{maxmu} we know that
$\mu=\lfloor\frac{d}{2}\rfloor$. We will describe a set of minimal
generators of $\kK$ by computing successive Morley forms ---as in
\cite{CD10}--- between two generators of the $\mu$-basis and $T_1F_2(\X)-T_2F_1(\X)$. There will also
be a moving conic that will come into play if $d$ is even.

We start with the following proposition, which will be useful in
the sequel.
\begin{proposition}\label{prop2}
If $2i+j<d$, then every nonzero element of $\kK_{i,j}$ is a
polynomial multiple of $T_1F_2(\X)-T_2F_1(\X)$.
\end{proposition}

\begin{proof}
Let  $P(\T,\X)\in \kK_{i,j}$. Due to (\ref{connect}) we have
$$F_2(\X)^iP(\T,\X)-T_2^iP(F_1(\X),F_2(\X),\X)=(T_1F_2(\X)-T_2F_1(\X))Q(\T,\X)
$$
with ---thanks to Proposition \ref{multiple}---
$P(F_1(\X),F_2(\X),\X)$ a homogeneous polynomial multiple of
$E(\X)$ of degree $2i+j<d=\deg(E(\X))$. As $E(\X)$ is
irreducible, we have then  $P(F_1(\X),F_2(\X),\X)=0$ and so
$$F_2(\X)^iP(\T,\X)=(T_1F_2(\X)-T_2F_1(\X))Q(\T,\X),
$$
which implies that there exists $Q_0(\T,\X)$ such that
$$P(\T,\X)=(T_1F_2(\T,\X)-T_2F_1(\T,\X))Q_0(\T,\X).$$
\end{proof}

\subsection{$d$ odd}\label{odd}
In this section we will assume $d=2k+1$. By Corollary \ref{maxmu},
we then have $\mu=k$. Let $\{p_{k,1}(\T,\X),\,q_{k+1,1}(\T,\X)\}$ be a 
basis of $\mbox{Syz}(I)$.
\smallskip
\begin{proposition}\label{polar}
Up to a nonzero constant in $\K$, we have that
\begin{equation}\label{jj}
p_{k,1}(F_1(\X),F_2(\X),\X)=E(\X).
\end{equation}
\end{proposition}
\begin{proof}
The polynomial $p_{k,1}(F_1(\X),F_2(\X),\X)$ is either identically zero
or has degree $d=\deg(E(\X))$ and, due to Proposition
\ref{multiple}, we know that it is a multiple of $E(\X)$. If we
show that it is not identically zero, then we are done. But if
this were not the case, due to (\ref{connect}) we would have to
conclude that $p_{k,1}(\T,\X)$ is a multiple of
$T_1F_2(\X)-T_2F_1(\X)$, which is impossible as the latter has
degree $2$ in the variables $\X$'s.
\end{proof}

\smallskip

We will define one nonzero element in $P_j(\T,\X)\in\kK_{j,d-2j}$
for $j=0,1,\ldots, k-1$. We will do this recursively starting from
$\kK_{k-1,2}$ and increasing the $\X$-degree at the cost of decreasing the $\T$-degree. This is the analogue of ``computing Sylvester forms'' in
\cite{cox08,bus09}, and we will perform essentially the same operations we have done in \cite{CD10} in order to get a minimal set of
generators of $\kK$ for curves parameterizable by lines.

Set then $P_k(\T,\X):=p_{k,1}(\T,\X);$ and for  $j$ from $0$ to $k-1$ do:
\begin{itemize}
\item[-] write $P_{k-j}(\T,\X)$ as
$A_{k-j}(\T,\X)T_1+B_{k-j}(\T,\X)T_2$ (clearly there is more than
one way of doing this, just choose one),  
\item[-] Set
$P_{k-j-1}(\T,\X):=A_{k-j}(\T,\X)F_1(\X)+B_{k-j}(\T,\X)F_2(\X)$.
\end{itemize}
We easily check that $P_j(\T,\X)\in \kK_{j,d-2j}$ for $j=0,\ldots,
k-1$, and also that (up to a nonzero constant in $\K$), 
\begin{equation}\label{remm}
P_j(F_1
(\X),F_2(\X),\X)=E(\X). 
\end{equation}
In addition, it is easy to check that
$P_0(\T,\X)=E(\X)$.

\begin{theorem}\label{mtodd}
A minimal set of generators of $\kK$ is
$$J:=\{T_1F_2(\X)-T_2F_1(\X),\,q_{k+1,1}(\T,\X),\,P_0(\T,\X),\ldots,\,P_k(\T,\X)\}.
$$
\end{theorem}

\begin{proof}
Let us first check that $J$ is a minimal set of generators of the
ideal generated by its elements. The forms
$T_1F_2(\X)-T_2F_1(\X)$,
$q_{k+1,1}(\T,\X)$,$P_k(\T,\X),\ldots,\,P_0(\T,\X)$  have
bidegrees $(1,2)$, $(k+1,1)$, $(k,1)$, $(k-1,3)$
$\dots$, $(1,2k-1)$, $(0,2k+1)$  respectively. Taking into account these
bidegrees we observe that, since $k\geq 2$, it is clear that
$T_1F_2(\X)-T_2F_1(\X)$ cannot be a polynomial combination of the
others. Also, $q_{k+1,1}(\T,\X)$ can only be a multiple of $P_k(\T,\X),$ which is impossible since they are a basis of $\mbox{Syz}(I)$.

Suppose now that $P_j(\T,\X)$ for some $j=0,\dots,k$ is a
polynomial combination of the others; then
\[
P_j(\T,\X)=H_0(\T,\X)(T_1F_2(\X)-T_2F_1(\X))+H_1(\T,\X)q_{k+1,1}(\T,\X)+\sum_{i\neq
j}G_i(\T,\X)P_i(\T,\X) \]
 All the elements $\{P_j(\T,\X)$,
$j=0,\dots ,k\}$, are nonzero and have different bidegrees $(j,d-2j).$ In addition, $\deg_{\T}(q_{k+1,1}(\T,\X))=k+1>j.$ Thus, 
$$H_1(\T,\X)= G_i(\T,\X)=0, \ i\neq j.$$ It remains
to show that $P_j(\T,\X)$ is not a multiple of
$T_1F_2(\X)-T_2F_1(\X)$. But if this were the case, then we would
have that $P_j(F_1(\X),F_2(\X),\X)=0$, in contradiction with \eqref{remm}. We conclude then that
$J$ is a set of minimal generators of $\langle J\rangle.$
\par
Now we have to show that $\langle J\rangle =\kK$, one of the inclusions being obvious. Let $P(\T,\X)$
be a nonzero element in $\kK_{i,j}$. If $2i+j<d$ then due to
Proposition \ref{prop2}, $P(\T,\X)$ is a multiple of
$T_1F_2(\X)-T_2F_1(\X)$ and the claim follows straightforwardly.
Suppose then $2i+j\geq d$. Thanks to Theorem \ref{mu} we only have
to look at  $0\leq i\leq k$. As $P(F_1(\X),F_2(\X),\X)=E(\X)h(\X)$
(due to Proposition \ref{multiple}), then by applying
\eqref{connect} to both $P(\T,\X)$ and $P_{i}(\T,\X)h(\X)$ we will
get
$$F_2(\X)^i\left(P(\T,\X)-P_i(\T,\X)h(\X)
\right)\in\langle T_1F_2(\X)-T_2F_1(\X)\rangle,
$$
and from here we deduce
$$P(\T,\X)\in\langle P_i(\T,\X),\,T_1F_2(\X)-T_2F_1(\X)\rangle\subset\langle J\rangle.$$
\end{proof}

\subsection{$d$ even}\label{even} 
Suppose now that $d=2k$. In this case we have again $\mu=k$, but also $d-\mu=k$ and hence there are two generators
of $\kK_{*,1}$ with $\T$-degree $k$. As usual, denote with $\{p_k(\T,\X),\,q_k(\T,\X)\}$
 a basis of $\mbox{Syz}(I)$. One can show easily now that there
exist nonzero linear forms $L_F(\X),\,L_G(\X)\in\K[\X]$ such that
$$p_k(F_1(\X),F_2(\X),\X)=E(\X)L_F(\X),\ \mbox{and}
\ q_k(F_1(\X),F_2(\X),\X)=E(\X)L_G(\X),$$
so we cannot use neither of these elements to get something like \eqref{jj}.
However, by applying some known results derived from the method of moving conics
explored in \cite{SGD97,ZCG99}, it turns out that we can find a polynomial in $\kK_{k-1,2}$ which will play the role of
$p_{k,1}(\T,\X)$ in Proposition \ref{polar} for this case.
\begin{proposition}\label{mconic}
There exists a nonzero element  $Q(\T,\X)\in \kK_{k-1,2}$ such
that  $$Q(F_1(\X),F_2(\X),\X)=E(\X).$$ Moreover, as $\K$-vector
spaces we have
\begin{equation}\label{sumdir}
 \kK_{k-1,2}=Q(\T,\X)\cdot\K \oplus\langle T_1F_2(\X)-T_2F_1(\X)\rangle_{k-1,2}.
\end{equation}
\end{proposition}

\begin{proof}
 In the language of moving curves, the fact that $d$ is even
and $\mu=k$ means that there are no moving lines of degree $k-1$
which follow the curve; that is, $\kK_{k-1,1}=0$. This condition
implies (see for instance Theorem $5.4$ in \cite{SGD97}) that
there exist $k$ linearly independent elements in $\kK_{k-1,2}$.
One can easily check that if we multiply $T_1F_2(\X)-T_2F_1(\X)$
by a polynomial $r(\T)$ of degree $k-2$, we then get an element of
$\kK_{k-1,2}$. The dimension of the $\K$-vector space generated by
all these polynomials is then $k-1$. Hence, there is one form
$Q(\T,\X)\in\kK_{k-1,2}$ which does not belong to this subspace,
and \eqref{sumdir} holds.
\par For this $Q(\T,\X)$ we easily get that $Q(F_1(\X),F_2(\X),\X)$ has to be a scalar multiple of $E(\X)$.
If it were zero,
then by using the same arguments as before we would have to
conclude that $Q(\T,\X)\in\langle T_1F_2(\X)-T_2F_1(\X)\rangle$,
which contradicts  \eqref{sumdir}.
\end{proof}

Now we will define nonzero elements in $\kK_{j,d-2j}$ for
$j=0,1,\ldots, k-1$. As before, we will do this recursively
starting from $\kK_{k-1,2}$ and increasing the $\X$-degree by
decreasing the $\T$-degree. Set $P_{k-1}(\T,\X):=Q(\T,\X)$ and,
for  $j$ from $0$ to $k-2$ do:
\begin{itemize}
\item[-] write $P_{k-1-j}(\T,\X)$ as
$A_{k-j}(\T,\X)T_1+B_{k-j}(\T,\X)T_2$ (there is more than
one way of doing this, just choose one), 
\item[-] Set
$P_{k-j-2}(\T,\X):=A_{k-j}(\T,\X)F_1(\X)+B_{k-j}(\T,\X)F_2(\X)$.
\end{itemize}
We easily check that $P_j(\T,\X)\in \kK_{j,d-2j}$ for $j=0,\ldots,
d-1$, and also that (up to a nonzero constant in $\K$),
\begin{equation}\label{remm2}
P_j(F_1(\X),F_2(\X),\X) = Q(F_1(\X),F_2(\X),\X) =E(\X)\ \forall
j=0,\ldots, k-1.
\end{equation} Also, by construction we have that
$P_0(\T,\X)=E(\X)$.

\begin{theorem}\label{mteven}
A minimal set of generators of $\kK$ is
$$J:=\{T_1F_2(\X)-T_2F_1(\X),\,P_0(\T,\X),\ldots,\,P_{k-1}(\T,\X),\,p_{k,1}(\T,\X),\,q_{k,1}(\T,\X)\}.
$$
\end{theorem}

\begin{proof}
As before, we  first check that $J$ is a minimal set of generators
of the ideal $\langle J\rangle$. We start again by verifying that
$p_k(\T,\X)$ and $q_k(\T,\X)$ cannot be combination of other
elements in the family due to the fact that they have minimal
$\X$-degree and $\K[\T]$-linearly independent. All the other elements
$P_j(\T,\X),\,j=0,\ldots, k-1$ are in different pieces of
bidegrees $(j,d-2j)$ so neither of them can be a polynomial
combination of the others. In addition, the form
$T_2F_1(\X)-T_1F_2(\X)$ is minimal with respect to the
$\T$-degree, so it is independent. It remains then show that
$P_j(\T,\X)$ is not a multiple of $T_1F_2(\X)-T_2F_1(\X)$. But if this
were the case, then we would have that
$P_j(F_1(\X),F_2(\X),\X)=0$, which contradicts \eqref{remm2}.
\par
In order to complete the proof, we must show that $\langle J\rangle =\kK$. As before, one of the inclusions is trivial.
Let then $P(\T,\X)$ be a nonzero element in $\kK_{i,j}$. If
$2i+j<d$ then due to Proposition \ref{prop2}, $P(\T,\X)$ is a
multiple of $T_1F_2(\X)-T_2F_1(\X)$ and the claim follows. Suppose
now $2i+j\geq d$. Thanks to Theorem \ref{mu} we only have to look
at  $0\leq i\leq k-1$. As $P(F_1(\X),F_2(\X),\X)=E(\X)h(\X)$ (due
to Proposition \ref{multiple}), then by applying \eqref{connect}
to both $P(\T,\X)$ and $P_{i}(\T,\X)h(\X)$ we will get
$$F_2(\X)^i\left(P(\T,\X)-P_i(\T,\X)h(\X)
\right)\in\langle T_1F_2(\X)-T_2F_1(\X)\rangle.
$$
From here we deduce $P(\T,\X)\in\langle
P_i(\T,\X),\,T_1F_2(\X)-T_2F_1(\X)\rangle\subset\langle J\rangle,$
and the claim follows.
\end{proof}

\begin{remark}
Note that the number of minimal generators in both cases $d=2k+1$
or $d=2k$ is always $k+3$, and also that a system of generators of
$\kK$ includes a $\K[\T]$-basis of $\mbox{Syz}(\I)$ and the implicit equation as expected.
\end{remark}

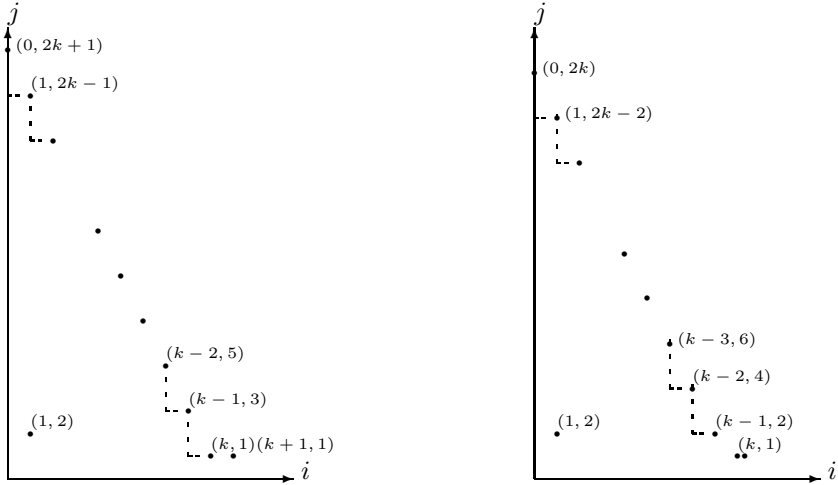
\setlength{\unitlength}{1mm}
\begin{picture}(120,65)
\put(0,0){\vector(1,0){38}}
 \put(39,0){$i$}
 \put(0,0){\vector(0,1){60}}
 \put(0,61){$j$}
 \put(3,6){\circle*{0.7} }
 \put(3,7){\tiny{$(1,2)$}}
\put(27,3){\circle*{0.7} }
 \put(27,4){\tiny{$(k,1)$}}
\put(24,9){\circle*{0.7} }
 \put(24,10){\tiny{$(k-1,3)$}}
\put(21,15){\circle*{0.7} }
 \put(21,16){\tiny{$(k-2,5)$}}
 \put(30,3){\circle*{0.7} }
 \put(33,4){\tiny{$(k+1,1)$}}
\put(0,57){\circle*{0.7}}
 \put(1,57){\tiny{$(0,2k+1)$}}
 \multiput(24,3)(1,0){2}{\line(1,0){0.5}}
 \multiput(24,3)(0,2){3}{\line(0,1){0.5}}
  \multiput(21,9)(1,0){2}{\line(1,0){0.5}}
 \multiput(21,9)(0,2){3}{\line(0,1){0.5}}
\multiput(0,51)(1,0){2}{\line(1,0){0.5}} \put(3,51){\circle*{0.7}
} \put(3,52){\tiny{$(1,2k-1)$}}
\multiput(3,45)(1,0){2}{\line(1,0){0.5}}
 \multiput(3,45)(0,2){3}{\line(0,1){0.5}}
\put(6,45){\circle*{0.7} }
 \put(12,33){\circle*{0.7} }
\put(15,27){\circle*{0.7} } \put(18,21){\circle*{0.7} }

\put(70,0){\vector(1,0){38}}
 \put(109,0){$i$}
 \put(70,0){\vector(0,1){60}}
 \put(70,61){$j$}
\put(73,6){\circle*{0.7} }
 \put(73,7){\tiny{$(1,2)$}}
\put(97,3){\circle*{0.7} }
 \put(97,4){\tiny{$(k,1)$}}
\put(98,3){\circle*{0.7}}
 \put(94,6){\circle*{0.7} }
 \put(94,7){\tiny{$(k-1,2)$}}
\put(91,12){\circle*{0.7} }
 \put(91,13){\tiny{$(k-2,4)$}}
 \put(88,18){\circle*{0.7} }
 \put(89,18){\tiny{$(k-3,6)$}}
\put(70,54){\circle*{0.7}}
 \put(71,54){\tiny{$(0,2k)$}}
 \multiput(91,6)(1,0){2}{\line(1,0){0.5}}
 \multiput(91,6)(0,2){4}{\line(0,1){0.5}}
  \multiput(88,12)(1,0){2}{\line(1,0){0.5}}
 \multiput(88,12)(0,2){4}{\line(0,1){0.5}}
\put(73,48){\circle*{0.7}} \put(74,48){\tiny{$(1,2k-2)$}}
 \put(76,42){\circle*{0.7}}
\multiput(70,48)(1,0){2}{\line(1,0){0.5}}
\multiput(73,42)(1,0){2}{\line(1,0){0.5}}
\multiput(73,42)(0,2){4}{\line(0,1){0.5}}
 \put(85,24){\circle*{0.7} }
  \put(82,30){\circle*{0.7} }

\end{picture}

\bigskip

\begin{example}
Consider the following parameterization
$$\left\{\begin{array}{ccl}
u_1(T_1,T_2)&=&T_1^5+T_2^5+T_1^4T_2\\
u_2(T_1,T_2)&=&T_1^3T_2^2\\
u_3(T_1,T_2)&=&T_1^5-T_2^5,
\end{array}\right.
$$ 
whose inverse can easily be found as
$$\F(\X)=(4 X_1^2  +X_2 X_1 +4 X_1 X_3 +16 X_2^2  +X_2 X_3,\,4 X_1^2 + 6 X_1 X_2 +  X_2^2  + 2 X_2  X_3 - 4 X_3^2).
$$
We have here $d=5,\, \mu=2$ and with the aid of a computer software find the following $\mu$-basis:
$$\begin{array}{ccl}
p_{2,1}(\T,\X)&=&    2 T_1^2  X_2 + T_2 T_1 X_2 - T_2^2  X_1 - T_2^2  X_3,\\
q_{3,1}(\T,\X)&=& 8 T_1^3  X_1 - 8 T_1^3  X_3 - 4 T_1^2  T_2 X_1 - 4 T_1^2  T_2 X_3 + 2 T_1 T_2^2  X_1 + T_2 ^2 T_1 X_2 \\
& &+ 2 T_1 T_2^2  X_3 - T_2^3  X_1 - 16 T_2^3  X_2 - T_2^3  X_3.
\end{array}
$$
Now we can perform the algorithm given in Section \ref{odd}, write
$$P_2(\T,\X):=p_{2,1}(\T,\X)=( 2 T_1  X_2 + T_2  X_2)T_1 +(- T_2  X_1 - T_2  X_3)T_2,$$
and set
$$
\begin{array}{ccl}
P_1(\T,\X)&=&( 2 T_1  X_2 + T_2  X_2)F_1(\X) +(- T_2  X_1 - T_2  X_3)F_2(\X)\\ 
&=&32 T_1X_2^3  + 8 T_1 X_1^2 X_2  + 2 T_1X_1 X_2^2  + 8 T_1X_1 X_2  X_3 + 2 T_1X_2^2  X_3  \\ && + 16 T_2X_2^3  - 2 T_2X_1^2  X_2  
- 4 T_2X_1 X_2  X_3 - 4 T_2X_1^3  + 4 T_2 X_1 X_3^2  \\ & &- 4 T_2 X_1^2  X_3 - 2 T_2X_2 X_3^2   
+ 4 T_2 X_3^3.
\end{array}
$$
We perform the same operations on $P_1(\T,\X)$ to get the implicit equation:
$$\begin{array}{ccl}
P_0(\T,\X)&=&16\big( -X_1^5  + 33 X_2^5  - X_1^4 X_3  + 3 X_1^2  X_2  X_3^2 + 16 X_1 X_2^3  X_3 \\ &&+ X_1^3 X_2  X_3 - X_3^5  - 4 X_2^3  X_3^2  - X_1  X_3^4 + 2 X_1^3  X_3^2  + 2 X_1^2  X_3^3  
\\ &&+ 20 X_1^2  X_2^3  + 6 X_2^4  X_3 + 10 X_1 X_2^4  + X_2 X_3^4  + 3 X_1  X_2 X_3^3\big).
\end{array}
$$
By Theorem \ref{mtodd}, a minimal set of generators of $\kK$ is given by the five polynomials $p_{2,1}(\T,\X),\,q_{3,1}(\T,\X),\,P_1(\T,\X),\,P_0(\T,\X)$ and 
$T_1F_2(\X)-T_2F_1(\X).$
\end{example}

\medskip
\begin{example}\label{d6}
Set $d=6$ and consider
$$\left\{\begin{array}{ccl}
u_1(T_1,T_2)&=&T_1^6+T_1^5T_2\\
u_2(T_1,T_2)&=&T_1^3T_2^3\\
u_3(T_1,T_2)&=&T_2^6.
\end{array}\right.
$$ 
By computing explicitly a Gr\"obner basis of $\ker(\mathfrak{h})$ we get $\mu=3$ and the following $\mu$-basis:
$$\begin{array}{ccl}
p_{3,1}(\T,\X)&=&    T_1^3X_3-T_2^3X_2,\\
q_{3,1}(\T,\X)&=& T_2^3X_1-T_1^3X_2-T_1^2T_2X_2.
\end{array}
$$
A quadratic inverse can be found also as part of the Gr\"obner basis of $\kK$:
$$T_1F_2(\X)-T_2F_1(\X)= T_1(X_1X_3-X_2^2)-T_2X_2^2,
$$
and we can also detect a moving conic of degree $2$ in $\T$ which is not a multiple of the latter:
$$Q(\T,\X)=    T_1^2  X_2 X_3 - T_2^2  X_1 X_3 + T_2^2  X_2^2.
$$
Now we have all the ingredients to start with the algorithm presented in Section \ref{even}: set $P_2(\T,X):=Q(\T,\X)$, and write
$$P_2(\T,\X)=(T_1  X_2 X_3)T_1 +(- T_2  X_1 X_3 + T_2  X_2^2)T_2.
$$
Then, we have 
$$\begin{array}{ccl}
P_1(\T,\X)&:=&(T_1  X_2 X_3)F_1(\X) +(- T_2  X_1 X_3 + T_2  X_2^2)F_2(\X)\\
&=&(T_1  X_2 X_3)X_2^2 +(- T_2  X_1 X_3 + T_2  X_2^2)(X_1X_3-X_2^2)\\
&=&X_2^3X_3T_1-(X_1X_3-X_2^2)^2T_2;
\end{array}
$$
and finally we get
$$\begin{array}{ccccl}
E(\X)&=&P_0(\T,\X)&:=&X_2^3X_3F_1(\X)-(X_1X_3-X_2^2)^2F_2(\X)\\
&&&=&X_2^5X_3-(X_1X_3-X_2^2)^3
\end{array}
$$
which is the implicit equation of the curve. Theorem \ref{mteven} tells us now that a minimal set of generators of $\kK$ is given by
$p_{3,1}(\T,\X),\,q_{3,1}(\T,\X),\,P_2(\T,\X),\,P_1(\T,\X),\,E(\X).
$
\end{example}

\bigskip
\section{Conclusions and Open Problems}\label{conclu}
We have described in detail the geometric features of rational curves parameterizable by conics and the algebraic aspects of their parameterizations.
It would be interesting to get a similar description of families  of curves parameterizable by forms of low degree. For simplicity, we will set our open questions and
remarks for curves parameterizable by cubics (i.e.\  $\deg(\F(\X))=3$ in \eqref{inverse}), but of course the interest is to get a description for
general curves of degree $d$ parameterizable by forms of degree $d',\,d'\ll d$.

\begin{itemize}
 \item In Proposition \ref{deg2} and Remark \ref{elunico} we have shown that the only pairs of quadratic forms $(F_1(X_1,X_2),\,F_2(X_1,X_2))$ without common factors not inducing a birational application $\C\dasharrow\P^1$
for any plane curve $\C$ are those such that $T_1F_2(X_1,X_2)-T_2F_1(X_1,X_2)$ defines a degenerate conic in $\P^1_{\K(\T)}$. Is there a geometric or algebraic
condition analogue to this for pairs of cubics in $\K(\X)$? 
\item The description of the family of all rational parameterizations induced by a given inverse map $\psi$ in Proposition \ref{paramparam} is based on the
fact that the every nondegenerate conic in any projective plane over an algebraically closed field is parameterizable. Which is the analogue of this
fact for cubics? What is the equivalent of ``nondegenerate  conic'' in the case of cubics? 
\item One can prove a more general statement in one of the directions of Theorem \ref{mtmt}: if $\Lambda:\P^2\dasharrow\P^2$ is a birational transformation 
whose inverse is given by three cubics $F_1(\X),\,F_2(\X),\,F_3(\X)$, and $\C_0$ is a curve parameterizable by lines with singularity at $(0:0:1)$, then $\overline{\Lambda(\C_0)}$
is a curve parameterizable by cubics. Are these all of them? Note that for a regular sequence of homogeneous forms $F_1(\X),\,F_2(\X)$ of degree $3,$
the variety $\V(\F(\X))$ has cardinality $9$ counted with multiplicities. It turns out that $F_1(\X),\,F_2(\X),\,F_3(\X)$ defines a birational map if and only if
$\V(F_1(\X),\,F_2(\X),\,F_3(\X))$ has cardinality $8$ (counted with multiplicities). But Cayley-Bacharach Theorem (\cite{EGH96}) implies that any form
$F_3(\X)$ of degree $3$ vanishing in all but one point of $\V(\F(\X))$ must vanish in all of them. So, in principle ``extending'' a general regular sequence of
two cubics to an automorphism of $\P^2$ given by cubics as it was done in Proposition \ref{quadratic}, cannot be done straightforwardly.
\item The computation of minimal generators of ${\mathfrak h}$ in Section \ref{Rees} involved one moving conic that follows the parameterization whose knowledge
comes from the method of moving conics. There is no known systematic method for moving cubics so far. How can we detect forms of lower degree in $\X$ in order
to produce elements like the $Q(\T,\X)$ described in Proposition \ref{mconic}?
\end{itemize}
We hope that we shall be able to answer these questions in future papers.

\end{document}